\documentclass[11pt,a4paper]{article}

\usepackage{amsmath}
\usepackage{amssymb}
\usepackage{textcomp}
\usepackage{color}
\usepackage{mathrsfs}

\newtheorem{theorem}{Theorem}[section]
\newtheorem{lemma}[theorem]{Lemma}
\newtheorem{proposition}[theorem]{Proposition}

\newtheorem{definition}[theorem]{Definition}

\newtheorem{corollary}[theorem]{Corollary}

\newcommand{\vb}{\vspace{3mm}}

\newcommand{\DD}{{\rm d}}

\newcommand{\eee}{\mathscr{E}}
\newcommand{\PP}{{\mathbb P}}

\newcommand{\rr}{\mathbb{R}}
\newcommand{\xx}{\mathbb{X}}
\newcommand{\one}{\mathbf{1}}
\newcommand{\pp}{\mathbb{P}}
\newcommand{\dd}{\mathrm{d}}
\newcommand{\ee}{\mathbb{E}}

\newcommand{\diag}{\mathrm{diag}}

\newenvironment{proof}[1][\proofname]{\par \normalfont \trivlist
 \item[\hskip\labelsep\itshape #1]\ignorespaces
}{
 \hspace*{\fill}$\Box$ \endtrivlist
}
\newcommand{\proofname}{{\bf Proof}}

\begin{document}
\title{{Large deviations for Markov-modulated diffusion processes with rapid switching}}

\author{Gang Huang, Michel Mandjes, Peter Spreij}
\maketitle
\begin{abstract} \noindent
In this paper, we study small noise asymptotics of Markov-modulated diffusion processes
in the regime that the modulating Markov chain is rapidly switching. 
We prove the joint sample-path large deviations principle
for the Markov-modulated diffusion process and the occupation measure of the Markov chain
(which evidently also yields the  large deviations principle for each of them separately by applying the contraction principle). 
The structure of the proof is such that we first prove exponential tightness, and then establish a local large deviations principle (where the latter part is split into proving the corresponding upper bound and lower bound). 

\vb

\vb

\noindent {\it Address.} Korteweg-de Vries Institute for Mathematics, University of Amsterdam, Science Park 904, 1098 XH Amsterdam, the Netherlands.

\vb

\noindent {\it Email}. {\footnotesize $\{${\tt g.huang|m.r.h.mandjes|spreij}$\}${\tt @uva.nl}}

\vb

\end{abstract}

\section{Introduction}
The setting studied in this paper is the following. We consider a complete probability space $(\Omega, \mathscr{F}, \mathbb{P} )$ with a filtration 
$\{\mathscr{F}_{t}\}_{t\in \rr_{+}}$, where $\rr_{+}:=[0, +\infty)$. $\mathscr{F}_{0}$ contains all the $\pp$-null sets of $\mathscr{F}$,
and $\{\mathscr{F}_{t}\}_{t\in \rr_{+}}$ is right continuous. Let $X_{t}$ be a finite-state time-homogeneous Markov chain with transition intensity matrix $Q$ 
and state space $\mathbb{S}:=\{1, \cdots, d\}$ for some $d\in \mathbb{N}$.
The Markov-modulated diffusion process is defined as the unique solution to
\begin{equation*}
M_{t}=M_{0}+\int_{0}^{t}b(X_{s}, M_{s})\DD s+\int_{0}^{t}\sigma(X_{s}, M_{s})\DD B_{s},
\end{equation*}
where $B_{t}$ is a standard Brownian motion. 
We assume that there exist $i, x$ such that $\sigma(i, x) \neq 0$ throughout this paper. 
The concept of Markov modulation is also known as `regime switching'; the Markov chain $X_t$ is often referred to as the `background process', or the `modulating Markov chain'.

The objective of this paper is to study the above stochastic differential equation under a particular parameter scaling. 
For a strictly positive (but typically small) $\epsilon$, we scale $Q$ to $Q/\epsilon=:Q^{\epsilon}$, and denote by $X_t^{\epsilon}$ the Markov chain with this transition intensity matrix $Q^{\epsilon}$.
If the expected number of jumps per unit time is $y$ for $X_{t}$, then the time-scaling entails that it is $y/ \epsilon$ for $X^{\epsilon}_{t}$. One could therefore say that the Markov chain has been sped up by a factor $\epsilon^{-1}$, and,
as a consequence, $X^{\epsilon}_{t}$ switches rapidly among its states when $\epsilon$ is small. 
A classical topic in large deviations theory, initiated by Freidlin and Wentzell \cite{Freidlin}, concerns small-noise large deviations. In this paper, we investigate how rapid-switching behavior of $X^{\epsilon}_{t}$ affects the small-noise asymptotics of 
$X^{\epsilon}_t$-modulated diffusion processes on the interval $[0, T]$ (for any fixed strictly positive $T$). 

Let us make the scaling regime considered more concrete now. Importantly, it concerns a scaling of the function $\sigma(\cdotp, \cdotp)$  to $\sqrt{\epsilon}\sigma(\cdotp, \cdotp)$ in the Markov-modulated diffusion, but at the same time  we speed up the Markovian background process in the way we described above. 
The resulting process $M^{\epsilon}_{t}$ is defined as the unique strong solution to 
\begin{equation} \label{def}
M^{\epsilon}_{t}=M^{\epsilon}_{0}+\int_{0}^{t}b(X_{s}^{\epsilon}, M_{s}^{\epsilon})\DD s+
\sqrt{\epsilon}\int_{0}^{t}\sigma(X_{s}^{\epsilon}, M_{s}^{\epsilon})\DD B_{s},
 \end{equation}
where we recall that $X^{\epsilon}_t$ has transition intensity matrix $Q^{\epsilon}.$
Focusing on the regime that $\epsilon\to0$, we call in the sequel $M_t^{\epsilon}$ the Markov-modulated diffusion process with rapid switching.
For simplicity, we will assume throughout this paper that $M^{\epsilon}_{0}\equiv0$, whereas
$X_{0}^{\epsilon}$ starts at an arbitrary $x\in \mathbb{S}$, for all $\epsilon$. 
When we write e.g.\ $\ee[M^{\epsilon}_{t}]$, this is to be understood as the expectation of $M^{\epsilon}_{t}$ with the above initial conditions.

Since $M_t^{\epsilon}$ evolves in the random environment of $X_t^{\epsilon}$, we need to design a coupling to separate the effects of the vanishing of the diffusion term and the fast varying of the Markov chain, but at the same time to keep track of both of them. Since the scaling $Q$ to $Q/\epsilon$ is equivalent to speeding up time by a factor $\epsilon^{-1}$, one could informally say that $X_t^{\epsilon}$ relates to a faster time scale than $M_t^{\epsilon}$, and therefore essentially exhibits  stationary behavior `around' this specific $t$. Then it is custom to consider the occupation measure of $X^{\epsilon}_{t}$,
which is defined on $\Omega\times [0, T] \times \mathbb{S}$ as
\begin{equation}
\nu^{\epsilon}(\omega; t, i)=\int_{0}^{t}\one_{\{X_{s}^{\epsilon}(\omega)=i\}}\DD s.
\end{equation} 
As its name suggests, $\nu^{\epsilon}(\cdot ; T, i)$ measures the time $X_t^{\epsilon}$ spends in state $i$ during the time interval $[0, T]$.
 Moreover, we can use the derivative of $\nu^{\epsilon}(t)$ to gauge the infinitesimal change of the occupation measure of $X_{t}^{\epsilon}$, at any $t\in[0,T]$. 
We thus construct a coupling $(M^{\epsilon}, \nu^{\epsilon})$, which is the main object studied in this paper.

A celebrated result in Donsker and Varadhan \cite{Donsker} concerns the large deviations principle (LDP) for $\nu^{1}(\omega; t, \cdot)/t$ as $t \rightarrow \infty$ (i.e., the LDP of the fraction of time spent in the individual states of the background process). The setting of the present paper,  however, involves the 
{\it sample-path} LDP for $\nu^{\epsilon}$ on $[0, T]$ as $\epsilon \rightarrow 0$.
More precisely, we define the image space $\mathbb{M}_T$ of $\nu^{\epsilon}$ restricted on $[0, T]$ as the space of functions $\nu$ on $[0, T]\times\mathbb{S}$ satisfying $\nu( t , i)=\int_{0}^{t} K_{\nu}(s, i)\DD s$,
where $\sum_{i=1}^{d} K_{\nu}(s, i) =1 $, $K_{\nu}(s, i) \geqslant 0$ for every $i \in \mathbb{S}, s\in [0, T]$, and $K_{\nu}(s, i)$ being Borel measurable with respect to $s$; $K_{\nu}$ is referred to as the {\it kernel} of $\nu$. The metric
on $\mathbb{M}_T$ is defined as
\[d_T(\mu, \nu)
=\sup_{0\leqslant t\leqslant T, i \in \mathbb{S}}\left|\int_{0}^{t} K_{\mu}(s, i) \DD s-\int_{0}^{t} K_{\nu}(s, i) \DD s\right|. \]
We can also view $\mathbb{M}_T$ as a subset of $\mathbb{C}_{[0, T]}(\rr^d)$ which is the space of $\rr^{d}$-valued continuous functions on $[0, T]$. In addition, the metric $d_T$ on $\mathbb{M}_T$ is equivalent to 
the uniform metric on $\mathbb{C}_{[0, T]}(\rr^d)$.

We also define $\mathbb{C}_{T}$ as the image space of $M^{\epsilon}$, which is the space of functions $f\in \mathbb{C}_{[0, T]}(\rr)$ and $f(0)=0$
equipped with the uniform metric
$\rho_{T}(f,g):=\sup_{0\leqslant t\leqslant T}|f(t)-g(t)|.$
The product metric $\rho_{T} \times d_T $ on  $\mathbb{C}_{T}\times \mathbb{M}_T$ is defined by
\[(\rho_{T} \times d_T ) ((\varphi, \nu), (\varphi', \nu')):=\rho_{T}(\varphi, \varphi')+d_T(\nu, \nu'), \:\:\:\:
\forall (\varphi, \nu), (\varphi', \nu')\in \mathbb{C}_{T}\times \mathbb{M}_T.\]
We denote by $\mathscr{B}(\mathbb{C}_{T}\times \mathbb{M}_T)$  the Borel $\sigma$-algebra generated by the topology
induced by $\rho_{T} \times d_T $.

The main result of this paper is the {\it joint} sample-path LDP for $(M^{\epsilon}, \nu^{\epsilon})$ on $\mathbb{C}_{T}\times \mathbb{M}_T$. The associated (joint) large deviations rate function is obtained in quite an explicit form. It is actually the  sum of two expressions 
that we introduce later in this paper, viz.\
(\ref{diff}), i.e., the rate function $I_T(\varphi,\nu)$ corresponding to  $M^{\epsilon}$, and (\ref{CTMC}), i.e., the rate function $\tilde I_T(\nu)$ corresponding to $\nu^{\epsilon}$.
Informed readers will recognize that these rate functions are variants of those for 
diffusion processes, as given in e.g.\ Freidlin and Wentzell \cite{Freidlin}, and for occupation measures of Markov 
processes, as given in e.g.\ Donsker and Varadhan \cite{Donsker} (where we remark again that the
result in \cite{Donsker} relates to $\nu^{1}(\omega; t, \cdot)/t$ for $t$ large, whereas our statement concerns  the sample paths of $\nu^{\epsilon}$).

One method of proving the LDP for a family of probability measures on a metric space, 
as was introduced in the seminal papers of Liptser and Pukhalskii \cite{Liptser} and Liptser \cite{Liptser1}, is to first prove exponential 
tightness, and then the local LDP (precise definitions of these notions will be given in the next section). Our work by and large follows this approach.  Importantly, the model
considered in Liptser \cite{Liptser1} is similar to ours, in that it also studies the stochastic differential equation
 (\ref{def}), but in the setup of Liptser \cite{Liptser1} the process $X^{\epsilon}_{t}$ is another diffusion process (rather than a finite-state Markov chain).  It means that we can roughly follow the structure of the proof presented in \cite{Liptser1} (we also rely on  the method of stochastic exponentials, for instance), but there are crucial differences at many places. For instance, as we point out below, there are several novelties that have the potential of being used in other settings, too. 

One of the methodological novelties is the following.
We explore a nice connection between regularity properties of the rate function $\tilde{I}_T(\nu)$ in the LDP for $(M^{\epsilon}, \nu^{\epsilon})$ and a dense subset of the image space $\mathbb{M}_T$ of $\nu^{\epsilon}$. On this dense subset, the optimizer of the integrand of $\tilde{I}_T(\nu)$ is infinitely differentiable. This eliminates many difficulties in the computation and leads us to first prove the local LDP on a dense subset of $\mathbb{C}_T \times \mathbb{M}_T$. We then extend the local LDP to $\mathbb{C}_T \times \mathbb{M}_T$ by continuity properties of the rate functions $I_T(\varphi,\nu)$ and $\tilde{I}_T(\nu)$.  

Let $\mathbb{U}$ denote the space of functions on $[0, T] \times \mathbb{S}$ 
being continuously differentiable on $[0, T]$ and $\inf_{s \in [0, T], i \in \mathbb{S}} u(s, i)>0$. In our analysis in Section 6, we identify the following stochastic
exponential which is directly related to the Markov chain $X^{\epsilon}_t$ and its 
rate function $\tilde{I}_T(\nu)$ (as given in (\ref{CTMC})):\[\frac{u(t, X^{\epsilon}_{t})}{u(0, X^{\epsilon}_{0})}\exp\left(-\int_{0}^{t}\frac{\frac{\partial}{\partial s} u(s, X^{\epsilon}_s)+(Q^{\epsilon}\,u)
(s, X^{\epsilon}_{s})}{u(s, X_{s}^{\epsilon})}\DD s\right), \:\:\: u \in \mathbb{U},\]
which plays a key role when proving the local LDP.
Here we follow the notational convention that
$(Q^\epsilon u)(s,i)=\sum_{j=1}^{d}Q^\epsilon_{ij}\,u(s,j),$ for $i \in \mathbb{S}$.

As mentioned above, the main result of our paper is the \emph{joint} sample-path LDP for $(M^{\epsilon}, \nu^{\epsilon})$.
The LDPs for each component $M^{\epsilon}$ and $\nu^{\epsilon}$ are then derived as corollaries from our main result in the standard way, i.e., by an application of the contraction principle.
The small noise LDP for the Markov-modulated diffusion processes (which is $M^{\epsilon}$ alone) is also studied in a newly published paper by He and Yin \cite{he3} in a setting of multi-dimensional processes and time-depending transition intensity matrices. In our corresponding result, which is Corollary \ref{modulation}, the rate function for $M^{\epsilon}$ is decomposed into two parts that allow an appealing interpretation: the first part  corresponds to the rare behavior of the background process $X^{\epsilon}$, where the second part corresponds to the rare behavior of $M^{\epsilon}$ (conditional on the rare behavior of $X^{\epsilon}$).  
The rate function in He and Yin \cite{he3} is less explicit, in that it is expressed in terms of an H-functional in which the aforementioned two parts cannot be distinguished.
The sample-path LDP for occupation measures of rapid switching Markov chain (which is $\nu^{\epsilon}$ alone) is obtained in Theorem 5.1 in He {\it et al.} \cite{he1}. The rate function, which is also expressed in terms of an H-functional, coincides with the rate function in our LDP for $\nu^{\epsilon}$ (Corollary \ref{occupation}) when the transition intensity matrix is time-homogeneous. However, focusing on obtaining the LDP for the Markov-modulated diffusion process together with the background process, our aim and approach in this paper are entirely different from theirs.

The large-deviations analysis for stochastic processes with Markov-modulation is a currently active research field. Besides the previously mentioned papers  of He {\it et al.} \cite{he1} and He and Yin \cite{he3}, we list a few more. Guillin \cite{guillin} proved the averaging principle (moderate deviations) 
of Equation (\ref{def}) where $X^{\epsilon}_{t}$ is an exponentially ergodic Markov process and $b, \sigma$ are bounded functions. He and Yin \cite{he2} studied  the moderate-deviations behavior of  $M^{\epsilon}_{t}$ in Equation (\ref{def}), where $\sigma \equiv 0$ and $X^{\epsilon}_{t}$ is a non-homogeneous 
Markov chain with two time-scales. Lasry and Lions \cite{Lions2} and Fourni{\'e} {\it et al.} \cite{Lions} considered large deviations for the hitting times of Markov-modulated diffusion processes with rapid switching.

Interestingly, the present paper relates to  our previous work \cite{Huang}. For ease ignoring the initial position, we there considered the Markov-modulated diffusion
$\check M_t^{\epsilon}$ described by
\[
\check M^{\epsilon}_{t}=\int_{0}^{t}b(X_{s}^{\epsilon}, \check M_{s}^{\epsilon})\DD s+
\int_{0}^{t}\sigma(X_{s}^{\epsilon}, \check M^{\epsilon}_{s})\DD B_{s}.
\]
In the regime $\epsilon\to 0$ the solutions of the stochastic differential equation
converge weakly to a (non-modulated) diffusion $\check M_t$ satisfying, with $\pi$ denoting the stationary distribution of $X_t^{\epsilon}$ (and hence also of $X_t$),
\[
\check M_{t}=\int_{0}^{t}\sum_{i=1}^{d}b(i, \check M_{s})\pi(i)\DD s+
\int_{0}^{t}\left(\sum_{i=1}^{d}\sigma^{2}(i, \check M_s)\pi(i)\right)^{1/2}\DD B_{s}.
\]
This result shows that,  when the background chain switches rapidly, it is hard to distinguish from observed data
a Markov-modulated diffusion process from an `ordinary'  diffusion. The work in the present paper, in contrast, indicates that no such property carries over to the large deviations. The impact of a fast switching background chain does appear in the small noise asymptotics, as shown in the LDPs in this paper.

We now describe the organization of our paper.
The structure of the paper is as follows. In Section~\ref{prem}, we introduce some preliminary results, definitions, 
and notation. In Section 3, we state the paper's main result and explain the steps of its proof. In Section 4,
exponential tightness of $(M^{\epsilon}, \nu^{\epsilon})$ is verified. We identify a dense subset of $\mathbb{C}_T \times \mathbb{M}_T$ in Section 5, and explore regularity properties of the rate function on it. The upper bound and lower bound of the local LDP
for $(M^{\epsilon}, \nu^{\epsilon})$ are proved in Sections 6 and  7, respectively. We present a number of technical lemmas  in the appendix.

\section{Preliminaries}\label{prem}
In this section we first provide the definitions of the LDP, exponential tightness and the local LDP,
and state a set  of related theorems that are relevant in the context of the paper. 
Let $\xx$ throughout denote a Polish space with Borel $\sigma$-algebra $\mathscr{B}(\xx)$ and a metric $\rho$.
\begin{definition} {\em (Varadhan \cite{Varadhan})}\:\label{LDP}
A family of 
probability measures $\pp^{\epsilon}$ on $(\xx, \mathscr{B}(\xx))$ is said to obey the LDP with a 
rate function $I(\cdot)$ if there exists a function $I(\cdot): \xx \rightarrow [0, \infty]$
satisfying:

\noindent
{\em (1)} There exists $x\in \xx$ such that $I(x)<\infty$; I is lower semicontinuous; 
for every $c<\infty$ the set $\{x:I(x)\leqslant c\}$ is a compact set in $\xx$.

\noindent {\em (2)} For every closed set $F \subset \xx$, $\limsup_{\epsilon \rightarrow 0} \epsilon \log\pp^{\epsilon}(F)\leqslant -\inf_{x \in F}I(x).$

\noindent
{\em (3)} For every open set $O \subset \xx$, $\liminf_{\epsilon \rightarrow 0} \epsilon \log\pp^{\epsilon}(O)\geqslant -\inf_{x \in O}I(x).$
\end{definition}

\begin{definition} {\em (Den Hollander \cite{Den hollander}, Puhalskii \cite{Puhalskii})}\:
A family of probability measures $\pp^{\epsilon}$ on $(\xx, \mathscr{B}(\xx))$
is said to be exponentially tight, if for every $L < \infty$,
there exists a compact set $K_{L} \subset \xx$ such that
\[\limsup_{\epsilon\rightarrow 0}\epsilon \log\pp^{\epsilon}(\xx \setminus K_L) \leqslant -L.\]
\end{definition}

\begin{definition} {\em (Puhalskii \cite{Puhalskii}, Liptser and Puhalskii \cite{Liptser1})}\:
A family of probability measures $\pp^{\epsilon}$ on $(\xx, \mathscr{B}(\xx))$
is said to obey the local LDP with a rate function $I( \cdot)$
if for every $x \in \xx$
\begin{equation} \label{upper}
\limsup_{\delta\rightarrow 0}\limsup_{\epsilon \rightarrow 0} \epsilon \log \PP^{\epsilon} (\{y \in \xx:\rho(x,y)\leqslant \delta\}) \leqslant -I(x), 
\end{equation}
\begin{equation} \label{lower}
\liminf_{\delta\rightarrow 0}\liminf_{\epsilon \rightarrow 0} \epsilon \log \PP^{\epsilon} (\{y \in \xx:\rho(x,y)\leqslant \delta\}) \geqslant -I(x).
\end{equation}
\end{definition}

Since $\xx$ is a Polish space,  Definition \ref{LDP}.(1) implies exponential tightness.
Also,  Definition \ref{LDP}.(2)--(3) guarantee that $\pp^{\epsilon}$ 
satisfies the local LDP. Actually, the converse is also valid and is the key to prove our main result.
 
\begin{theorem} {\em (Puhalskii \cite{Puhalskii}, Liptser and Puhalskii \cite{Liptser1})}\:\label{prove}
If a family of probability measures $\pp^{\epsilon}$ on $(\xx, \mathscr{B}(\xx))$
is exponentially tight and obeys the local LDP 
with a rate function $I$, then it
obeys the LDP with the rate function $I$.
\end{theorem}

The following lemma, which corresponds to Lemma 1.4 in Borovkov and Mogulski{\u\i} \cite{borovkov}, shows that a local LDP on a dense subset of $\xx$ is enough for the validation of the local LDP on $\xx$, provided the rate function possesses a regularity property.
\begin{lemma} \label{mogu}
(i) If (\ref{upper}) is fulfilled for all $\tilde{x} \in \tilde{\mathbb{X}}$, where $\tilde{\mathbb{X}}$ is dense in $\mathbb{X}$ and function $I(x)$ is lower semi-continuous, then it holds for all $x \in \mathbb{X}$.\\
(ii) If for every $x \in \mathbb{X}$ with $I(x) < \infty$ there exists a sequence $\tilde{x}_n \in \tilde{\mathbb{X}}$ converging to $x$ and $I(\tilde{x}_n) \rightarrow I(x)$, then the fullfillment of 
(\ref{lower}) for $\tilde{x} \in \tilde{\mathbb{X}}$ implies the same for all $x \in \mathbb{X}$.
\end{lemma}

Next we impose some assumptions on the stochastic differential equation (\ref{def}), as was defined in the introduction.
{It is noted that (A.1) (`Lipschitz continuity') implies (A.2) (`linear growth'); we chose to include (A.2) as well, however, for ease reference in later  sections.}

\begin{itemize}
\item[(A.1)]{\it Lipschitz continuity}: there is a positive constant $K$ such that 
\[|b(i, x)-b(i, y)|+ |\sigma(i, x)-\sigma(i, y)| \leqslant K|x-y|, \:\:\:\: \forall i\in \mathbb{S}, \:\:\:x, y\in \rr.\]
\item[(A.2)]{\it Linear growth}: there exists a positive constant $K$ (which might be different from the $K$ used in (A.1)) such that
\[|b(i,x)|+|\sigma(i,x)|\leqslant K(1+|x|), \:\:\:\: \forall i\in \mathbb{S}, \:\:\:x\in \rr.\]
\item[(A.3)]{\it Independence}: the Markov chain $X_{t}^{\epsilon}$ is independent of the Brownian motion $B_{t}$ for all $\epsilon$.
\item[(A.4)]{\it Irreducibility}: the off-diagonal entries of the transition intensity matrix $Q$ are strictly positive. Hence, the Markov chain $X_{t}^{\epsilon}$ is irreducible for all $\epsilon$ and has an invariant probability measure $\pi=(\pi(1), \cdots, \pi(d)).$
\end{itemize}

Finally, we introduce some extra notation and function spaces. For an arbitrary stochastic process or a function $Y_{t}$, we denote the running maximum process by
$Y^{*}_{t}:=\sup_{s\leqslant t}|Y_{s}|$. For a semimartingale $Y_{t}$ such that $Y_{0}=0$, its {\it stochastic exponential} is defined as a semimartingale $\eee(Y)_{t}$ 
which is the unique strong solution to 
\[\eee(Y)_t=1+\int_{0}^{t}\eee(Y)_{s-}\dd Y_s.\]
We denote $\mathbb{H}_{T}$ the Cameron-Martin space of functions 
$\varphi \in \mathbb{C}_{T}$ such that $\varphi(t)=\int_{0}^{t} \varphi'(s) \DD s$ and $\varphi'$ is square-integrable on $[0, T]$. We call $\varphi'$ the derivative of $\varphi$.

\section{Main results}
We first introduce the definitions of the rate functions involved in the main result. The rate function corresponding to $\nu^{\epsilon}$ is defined as
\begin{equation} \label{CTMC}
 \tilde{I}_{T}(\nu):=\int_0^T
 \sup_{u \in U}\left[-\sum_{i=1}^{d}\frac{(Qu)(i)}{u(i)}K_{\nu}(s, i)\right] \DD s, \:\:\:\: \nu \in \mathbb{M}_T,
\end{equation}
where we recall the notation
$(Qu)(i)=\sum_{j=1}^{d}Q_{ij}u(j),$ for $i \in \mathbb{S}$,
and 
$U$ denotes the set of $d$-dimensional component-wise strictly positive vectors.
We now define the rate function corresponding to $M^{\epsilon}$.
For any $(\varphi, \nu) \in \mathbb{C}_{T}\times \mathbb{M}_T$, we define
\begin{equation}\label{diff}
 I_{T}(\varphi, \nu) := \left\{
 {\displaystyle 
  \begin{array}{l l}{\displaystyle 
 \frac{1}{2}\int_{0}^{T} \frac{[\varphi'_{t}-\hat{b}(\nu, \varphi_{t})]^2}{\hat{\sigma}^2(\nu, \varphi_{t})} \DD t }& \quad 
 \text{if } \varphi \in \mathbb{H}_{T},\\
    \infty & \quad \text{otherwise.}
  \end{array} }\right.
\end{equation}
where
\[\hat{b}(\nu, x):= \sum_{i=1}^{d} b(i,x)K_{\nu}(t, i),\:\:\:\:\: 
\hat{\sigma}(\nu, x):=\left(\sum_{i=1}^{d}\sigma^{2}(i, x)K_{\nu}(t, i)\right)^{1/2}.\]
In the above formulae, we follow the conventions that $0/0=0$ and $n/0=\infty,$ for all $n>0$.  When we fix a time $T$,  $(M^{\epsilon}, \nu^{\epsilon})$ is understood as a joint process restricted on $[0, T]$.
Let $\pp \circ (M^{\epsilon}, \nu^{\epsilon})^{-1}$ denote 
$\pp((M^{\epsilon}, \nu^{\epsilon})\in \cdot)$, which
is a family of probability measures on $(\mathbb{C}_{T}\times \mathbb{M}_T, \mathscr{B}(\mathbb{C}_{T}\times \mathbb{M}_T))$.
Also, $\pp \circ (M^{\epsilon})^{-1}$ and $\pp \circ (\nu^{\epsilon})^{-1}$ are families of
probability measures on  $(\mathbb{C}_{T}, \mathscr{B}(\mathbb{C}_{T}))$ and  $(\mathbb{M}_T, \mathscr{B}(\mathbb{M}_T))$
respectively. The following theorem is our main result which states the joint sample-path LDP of ($M^{\epsilon}, \nu^{\epsilon})$ on $[0, T]$, as $\epsilon\to 0$..
\begin{theorem} \label{mainresult}
For every $T>0$, the family  $\pp \circ (M^{\epsilon}, \nu^{\epsilon})^{-1}$
obeys the LDP in $(\mathbb{C}_{T}\times \mathbb{M}_T, \rho_T \times d_T)$ with the rate function
\[L_{T}(\varphi, \nu)=I_{T}(\varphi, \nu)+\tilde{I}_{T}(\nu).\]
\end{theorem}
\begin{proof}
The proof relies on applying Theorem \ref{prove}. We first need to prove the exponential tightness of 
$\pp \circ (M^{\epsilon}, \nu^{\epsilon})^{-1}$ on 
$(\mathbb{C}_{T}\times \mathbb{M}_T, \mathscr{B}(\mathbb{C}_{T}\times \mathbb{M}_T))$, i.e.,
for every $L < \infty$, there exists a compact set $K_{L} \subset \mathbb{C}_{T}\times \mathbb{M}_T$ such that
\[\limsup_{\epsilon\rightarrow 0}\epsilon \log\pp\left((M^{\epsilon}, \nu^{\epsilon})\in
\mathbb{C}_{T}\times \mathbb{M}_T\setminus K_L\right) \leqslant -L.\]
It is obvious that $\pp \circ (M^{\epsilon}, \nu^{\epsilon})^{-1}$ 
is exponentially tight if so are $\pp \circ (M^{\epsilon})^{-1}$ and 
$\pp \circ (\nu^{\epsilon})^{-1}$. As we mentioned earlier, $\mathbb{M}_T$ is a subset of $\mathbb{C}_{[0, T]}(\rr^d)$. For any $\nu \in \mathbb{M}_T$, its derivative
$K_{\nu}(s,i)$ is bounded by 1. Then all $\nu \in \mathbb{M}_T$ have the same Lipschitz constant, and hence $\mathbb{M}_T$ is equicontinuous. It is easily seen that $\mathbb{M}_T$ is bounded and closed. Then the Arzel\`a-Ascoli theorem implies that $\mathbb{M}_T$ is compact. The
exponential tightness of $\pp \circ (\nu^{\epsilon})^{-1}$ is satisfied since we can take $K_L=\mathbb{M}_T$.  
Exponential tightness of $\pp \circ (M^{\epsilon})^{-1}$ is 
verified in Proposition \ref{expotight} below.

Secondly, we proceed to prove that $\pp \circ (M^{\epsilon}, \nu^{\epsilon})^{-1}$
obeys the local LDP with the rate function $L_{T}(\varphi, \nu)$. That is, 
for every $(\varphi, \nu) \in \mathbb{C}_{T}\times \mathbb{M}_T$, we need to obtain the upper bound
\[\limsup_{\delta\rightarrow 0}\limsup_{\epsilon \rightarrow 0} \epsilon \log \PP
(\rho_T(M^{\epsilon}, \varphi)+d_T(\nu^{\epsilon}, \nu)\leqslant \delta)\leqslant -L_{T}(\varphi, \nu),\]
and the lower bound
\[\liminf_{\delta\rightarrow 0}\liminf_{\epsilon \rightarrow 0} \epsilon \log \PP
(\rho_T(M^{\epsilon}, \varphi)+d_T(\nu^{\epsilon}, \nu)\leqslant \delta)\geqslant -L_T(\varphi, \nu).\]
The core of the proof is proving the local LDP on a dense subset of $\mathbb{C}_{T}\times \mathbb{M}_T$. 
The upper bound is validated in Proposition \ref{upp}.
The lower bound is first proved in Proposition \ref{non} given the condition $\inf_{i, x}\sigma^2(i , x) > 0$. 
Then the condition is lifted in Proposition \ref{low} by a perturbation argument.
\end{proof}

The LDP for $\pp \circ (M^{\epsilon})^{-1}$ only (rather than for $\pp \circ (M^{\epsilon}, \nu^{\epsilon})^{-1}$)
is then derived from Theorem\ \ref{mainresult} by the contraction
principle in Dembo and Zeitouni \cite{Dembo}. We follow the convention that $\inf(\emptyset)=\infty$.
\begin{corollary} \label{modulation}
The family $\pp \circ (M^{\epsilon})^{-1}$
obeys the LDP with the rate function $\inf_{\nu \in \mathbb{M}_T} L_{T}(\varphi, \nu)$.
\end{corollary}

At an intuitive level, $\tilde{I}_{T}(\nu)$ can be interpreted as the `cost' of forcing $\nu^{\epsilon}$
to behave like $\nu$ on $[0, T]$. The other term $I_{T}(\varphi, \nu)$, can be seen as the `cost' of 
the sample paths of $M^{\epsilon}$ being close to $\varphi$ \emph{conditional on} $\nu^{\epsilon}$
{\it behaving like $\nu$ on $[0, T]$}. Then $\inf_{\nu \in \mathbb{M}} L_{T}(\varphi, \nu)$ indicates
the {\it minimal} \,{`cost'} of  the sample paths of $M^{\epsilon}$ being close to $\varphi$  on $[0, T]$.

Suppose $F$ is a closed or an open subset of $\mathbb{C}_{T}$. We can also interpret Corollary \ref{modulation} 
as the concentration of the probability $\pp \circ (M^{\epsilon})^{-1}(F)$, which is the set of sample paths of $M^{\epsilon}$,
on the `most likely path'
$\arg\inf_{\varphi \in F}(\inf_{\nu \in \mathbb{M}_T}  [I_{T}(\varphi, \nu)+\tilde{I}_{T}(\nu)]).$
So there are two sources contributing to the large deviations behavior of $M^{\epsilon}$:
$I_{T}(\varphi, \nu)$ represents the contribution resulting from the small noise, and $\tilde{I}_{T}(\nu)$ represents
the one from the rapid switching of the modulating Markov chain.

Again by the contraction principle,   $\pp \circ (\nu^{\epsilon})^{-1}$
obeys the LDP in $(\mathbb{M}_T, d_T)$ with the rate function $\inf_{\varphi \in \mathbb{C}_T} I_T(\varphi, \nu)+\tilde{I}_T(\nu)$.
Since there exists a $\varphi \in \mathbb{H}_T$ such that $\varphi'_{t}=\hat{b}(\nu, \varphi_{t})$ for all $t \in [0, T]$ and all $\nu \in \mathbb{M}_T$, it immediately follows that
 $\inf_{\varphi \in \mathbb{C}_T} I_T(\varphi, \nu)=0$. Hence, we have the following corollary. 
\begin{corollary} \label{occupation}
The family  $\pp \circ (\nu^{\epsilon})^{-1}$
obeys the LDP in $(\mathbb{M}_T, d_T)$ with the rate function
$\tilde{I}_{T}(\nu)$.
\end{corollary}

\section{Exponential tightness}
We show the exponential tightness of $\pp \circ (M^{\epsilon})^{-1}$ 
by Aldous-Pukhalskii-type sufficient conditions, as dealt with in e.g.\ Aldous \cite{Aldous}, Liptser and Pukhalskii \cite{Liptser}.  
The following criterion for exponential tightness in $\mathbb{C}_T$, as well as an auxiliary lemma, are adapted from Theorem 3.1 and Lemma 3.1 in Liptser and Pukhaskii \cite{Liptser} (which
consider
c\`{a}dl\`{a}g processes 
with jumps) to our setting of continuous processes. Let $\Gamma_{T}(\mathscr{F}_t)$ denote the family of stopping times adapted to $\mathscr{F}_t$ taking values in $[0, T]$.

\begin{theorem} \label{Aldous-pukhaskii}
Let $Y^{\epsilon}_{t}: (\Omega,\{\mathcal{F}_{t}\}_{t\leqslant T},\pp )\rightarrow \mathbb{C}_T$. If
\begin{itemize}
\item[(i)]
\[ \lim_{K'\rightarrow \infty} \limsup_{\epsilon \rightarrow 0} \epsilon \log \mathbb{P}\left(Y_T^{\epsilon*}\geqslant K'\right)=-\infty,\]
\item[(ii)] \[\lim_{\delta \rightarrow 0} \limsup_{\epsilon \rightarrow 0} \epsilon \log \sup_{\tau \in \Gamma_{T}(\mathcal{F}_t)}
 \mathbb{P}\left(\sup_{t\leqslant \delta}|Y_{\tau+t}^{\epsilon}-Y_{\tau}^{\epsilon}|\geqslant \eta\right)=-\infty, \:\:\: \forall \eta >0, \]\end{itemize}
 then  $\pp \circ (Y^{\epsilon})^{-1}$ is exponentially tight. 
\end{theorem}

\begin{lemma} \label{expoineq}
Let $Y=(Y_{t})_{t\geqslant 0}$ be a continuous semimartingale with $Y_{0}=0$. Let $D$ denote the part corresponding to a
predictable process of locally bounded variation, and $V$ the part corresponding
to the quadratic variation of the local martingale. Assume that for $T>0$ there exists a convex function $H(\lambda), \lambda \in \rr$ with $H(0)=0$ and such that for all $\lambda \in \rr$
and $t\leqslant T$
\[\lambda {D}_{t}+\lambda^{2}V_{t}/2\leqslant tH(\lambda \xi), \text{ a.s.},\]
where $\xi$ is a nonnegative random variable defined
on the same probability space as $Y$. Then, for all $c>0$ and $\eta>0$,
\[\mathbb{P}(Y_{T}^{*}\geqslant \eta)\leqslant \mathbb{P}(\xi> c)+\exp\left \{-\sup_{\lambda \in R}[\lambda\eta-TH(\lambda c)]\right \}.\]
\end{lemma}

We are now ready to prove the exponential tightness claim. The technique borrows elements from Liptser \cite{Liptser1}.

\begin{proposition} \label{expotight}
For every $T>0$, the family $\pp \circ (M^{\epsilon})^{-1}$ is exponentially tight on $(\mathbb{C}_T, \mathscr{B}(\mathbb{C}_T))$. 
\end{proposition}
\begin{proof}  
Firstly, we verify the condition (i) of Theorem \ref{Aldous-pukhaskii} for the process $M_{T}^{\epsilon*}$. For any $T>0$, evidently,
\[
M_{T}^{\epsilon*}\leqslant 
\int_{0}^{T}|b(X_{s}^{\epsilon},M_{s}^{\epsilon})|\DD s
+\sup_{t\leqslant T}\left|\sqrt{\epsilon}\int_{0}^{t}\sigma(X_{s}^{\epsilon}, M_{s}^{\epsilon})\DD B_{s}\right|, \:\:\: \text{a.s..}
\]
We denote $C_{t}^{\epsilon}:=\sqrt{\epsilon}\int_{0}^{t}\sigma(X_{s}^{\epsilon}, M_{s}^{\epsilon})\DD B_{s}$. By (A.2),
\[M_{T}^{\epsilon*}\leqslant K\int_{0}^{T} (1+M_{s}^{\epsilon*})\DD s
 +C_{T}^{\epsilon*}=KT+C_{T}^{\epsilon*}+K\int_{0}^{T} M_{s}^{\epsilon*}\DD s, \:\:\: \text{a.s..}\]
Since $KT+C_{T}^{\epsilon*}$ is nonnegative and non-decreasing in $T$, Gronwall's inequality implies
\begin{equation}\label{gronwall}
M_{T}^{\epsilon*}\leqslant e^{K T}\left[K T+C_{T}^{\epsilon*}\right], \:\:\: \text{a.s..}
\end{equation}
Now define $j_{K'}:=K'\exp(-K T)-K T$. Then (\ref{gronwall}) entails that for sufficiently large $K'$ such that $j_{K'}>0$, 
\[\mathbb{P}(M_{T}^{\epsilon*}\geqslant K') \leqslant \mathbb{P}(C_{T}^{\epsilon*}\geqslant j_{K'})\leqslant j_{K'}^{-1/ \epsilon}\,\ee\left[(C_{T}^{\epsilon*})^{1/ \epsilon}\right],\]
using Chebyshev's inequality. We thus conclude
\begin{equation}\label{log}
\epsilon \log \mathbb{P}(M_{T}^{\epsilon*}\geqslant K') \leqslant -\log  j_{K'}+\epsilon \log \ee\left[(C_{T}^{\epsilon*})^{1/ \epsilon}\right].
\end{equation}

We assume that $1/\epsilon >2$ in the rest of the proof (justified by the fact that we consider the limit $\epsilon\to0$). Since $C^{\epsilon}_t$ is a local martingale, the process
$|C^{\epsilon}_t|^{1/\epsilon}$ has a unique Doob-Meyer decomposition; let $\check C^{\epsilon}_t$
denote the unique predictable increasing process in this decomposition.
Applying a local martingale maximal inequality (see e.g.\ Liptser and Shiryaev \cite[Thm.\ 1.9.2]{liptser2}) to $C_{t}^{\epsilon}$, 
we have for the running maximum process that
\begin{equation} \label{local maximal}
\ee\left[(C_{T}^{\epsilon*})^{1/ \epsilon}\right]
\leqslant \left(\frac{1}{1-\epsilon}\right)^{1/\epsilon}\ee\left[\check C^{\epsilon}_T\right].
\end{equation}

In order to obtain an explicit expression for $\check C^{\epsilon}_t$, 
we apply It\^{o}'s formula to $|C_{t}^{\epsilon}|^{1/\epsilon}$. This means that, for any $t\in [0, T]$,
\[|C_{t}^{\epsilon}|^{1/\epsilon}=\frac{1}{\sqrt{\epsilon}}\int_{0}^{t}|C_{s}^{\epsilon}|^{1/\epsilon-1}\text{sign}(C_{s}^{\epsilon}) \sigma(X_{s}^{\epsilon}, M_{s}^{\epsilon})\DD B_{s}+
\frac{1-\epsilon}{2 \epsilon}\int_{0}^{t}|C_{s}^{\epsilon}|^{1/\epsilon-2} \sigma^{2}(X_{s}^{\epsilon}, M_{s}^{\epsilon})\DD s.
\]
We notice that the first part is a local martingale and the second part is a predictable increasing process. 
As a consequence,
\begin{equation}\label{predictable}
\check C^{\epsilon}_T=\frac{1-\epsilon}{2 \epsilon}\int_{0}^{T}|C_{s}^{\epsilon}|^{1/\epsilon-2} \sigma^{2}(X_{s}^{\epsilon}, M_{s}^{\epsilon})\DD s.
\end{equation}
Invoking (A.2) again, we have that $\sigma^{2}(X_{s}^{\epsilon},M_{s}^{\epsilon})\leqslant K^2(1+M_{s}^{\epsilon*})^2.$
Since (\ref{gronwall}) remains valid when replacing $T$ by $s$, for any $s\leqslant T$, we find
\begin{eqnarray*}
|C_{s}^{\epsilon}|^{1/\epsilon-2} \sigma^{2}(X_{s}^{\epsilon}, M_{s}^{\epsilon})
&\leqslant&  (C_{s}^{\epsilon*})^{1/\epsilon-2} K^2\left[1+e^{K s}(K s+C_{s}^{\epsilon*})\right]^2\\
&\leqslant&  (C_{s}^{\epsilon*})^{1/\epsilon-2}K^2\left[1+e^{K T}(K T+C_{s}^{\epsilon*})\right]^2\\
&\leqslant&  (C_{s}^{\epsilon*})^{1/\epsilon-2}K^2[2(1+e^{K T}K T)^2+2e^{2KT}(C_{s}^{\epsilon*})^2].
\end{eqnarray*} 
Let $ L_{T, K}=2K^2\max \left\{(1+e^{K T}K T)^2, e^{2KT}\right\}$. Then 
\[|C_{s}^{\epsilon}|^{1/\epsilon-2} \sigma^{2}(X_{s}^{\epsilon}, M_{s}^{\epsilon})
\leqslant  (C_{s}^{\epsilon*})^{1/\epsilon-2} L_{T, K}\left[1+(C_{s}^{\epsilon*})^2\right]
\leqslant  L'_{T, K}\left[1+(C_{s}^{\epsilon*})^{1/ \epsilon}\right],\]
where $L\equiv L'_{T, K}$ is a positive constant not depending on $K'$ (nor $\epsilon$). 
We plug (\ref{predictable}) and the above estimate into (\ref{local maximal}), so as to obtain
\begin{eqnarray*}
\ee[(C_{T}^{\epsilon*})^{1/ \epsilon}]
&\leqslant& \left(\frac{1}{1-\epsilon}\right)^{1/\epsilon}\ee\left[\frac{1-\epsilon}
{2 \epsilon}\int_{0}^{T}|C_{s}^{\epsilon}|^{1/\epsilon-2} \sigma^{2}(X_{s}^{\epsilon}, M_{s}^{\epsilon})\DD s\right]\\
&\leqslant& \left(\frac{1}{1-\epsilon}\right)^{1/\epsilon-1}\frac{L}{2\epsilon}\left(T+\int_{0}^{T}\ee\left[(C_{s}^{\epsilon*})^{1/ \epsilon}\right]\DD s\right)\\
&\leqslant& \left(\frac{1}{1-\epsilon}\right)^{1/\epsilon-1}\frac{LT}{2\epsilon} \,\exp{\left[\left(\frac{1}{1-\epsilon}\right)^{1/\epsilon-1}\frac{LT}{2\epsilon}\right]},
\end{eqnarray*}
the last inequality following from Gronwall's inequality. 
Now observe that $(1-\epsilon)^{1-1/\epsilon}$ is decreasing on $\epsilon \in[0,\frac{1}{2})$, with limiting value $e$ as $\epsilon\to 0$.
As a result, we have the following upper bound on the exponential decay rate of $\ee[(C_{T}^{\epsilon*})^{1/ \epsilon}]$:
\begin{eqnarray*}
\limsup_{\epsilon \rightarrow 0}\epsilon \log \ee\left[(C_{T}^{\epsilon*})^{1/ \epsilon}\right]
&\leqslant& \limsup_{\epsilon \rightarrow 0}
\left[\epsilon \log\left(\frac{1}{1-\epsilon}\right)^{1/\epsilon-1}
+\epsilon \log \frac{LT}{2\epsilon}
+\left(\frac{1}{1-\epsilon}\right)^{1/\epsilon-1}\frac{LT}{2}\right]\\
&\leqslant& \frac{e LT }{2} < \infty.
\end{eqnarray*}
Hence, by (\ref{log}), for all $T>0$, condition (i) of Thm.\ \ref{Aldous-pukhaskii} follows for the process $M_{T}^{\epsilon*}$:
\begin{equation}\label{isup}
        \lim_{K' \rightarrow \infty} \limsup_{\epsilon \rightarrow 0} \epsilon \log \mathbb{P}\left(M_{T}^{\epsilon*}\geqslant K'\right)=-\infty.
 \end{equation}

Secondly, we verify condition (ii) of Theorem \ref{Aldous-pukhaskii}. To this end, note that for arbitrary $T>0$, $\delta \leqslant 1$, and stopping time 
$\tau \in \Gamma_{T}(\mathcal{F}_t)$,
\begin{equation}\label{sep}
\mathbb{P}\left(\sup_{t\leqslant \delta}|M_{\tau+t}^{\epsilon}-M_{\tau}^{\epsilon}|\geqslant \eta \right)
\leqslant  
\mathbb{P}\left(\sup_{t\leqslant \delta}(M_{\tau+t}^{\epsilon}-M_{\tau}^{\epsilon})\geqslant \eta \right)+\mathbb{P}\left(\sup_{t\leqslant \delta}(M_{\tau}^{\epsilon}-M_{\tau+t}^{\epsilon})\geqslant \eta \right). 
\end{equation}
We can see that $M_{\tau+t}^{\epsilon}-M_{\tau}^{\epsilon}$ is a semimartingale with respect to the filtration $\{\mathcal{F}_{\tau+t}\}_{t\geqslant 0}$. For any $\tau \in \Gamma_{T}(\mathcal{F}_t)$, we denote
\[D_{t}^{\epsilon}:=\int_{\tau}^{\tau+t}b(X_{s}^{\epsilon}, M_{s}^{\epsilon})\DD s,\:\:\:\:
V_{t}^{\epsilon}:=\epsilon\int_{\tau}^{\tau+t}\sigma^{2}(X_{s}^{\epsilon}, M_{s}^{\epsilon})\DD s.\]
By (A.2), we have, for all $\lambda \in \rr$, $t\leqslant \delta \leqslant 1$ and $\tau \leqslant T$,
\[\lambda D_{t}^{\epsilon}+\frac{\lambda^2}{2}V_{t}^{\epsilon} \leqslant  |\lambda|K(1+M_{T+1}^{\epsilon*})t+\frac{\lambda^2\epsilon}{2}K^2(1+M_{T+1}^{\epsilon*})^2t,
\:\:\:\:\: \text{a.s..}\]
We define
\[H(\lambda):=|\lambda|+\frac{\lambda^2 \epsilon}{2},\:\:\:\:\:
\xi:=K(1+M_{T+1}^{\epsilon*}).\]
Then $M_{\tau+t}^{\epsilon}-M_{\tau}^{\epsilon}$ satisfies the conditions of Lemma \ref{expoineq} and,  for all $c>0, \eta>0$,
\[\mathbb{P}\left(\sup_{t\leqslant \delta}(M_{\tau+t}^{\epsilon}-M_{\tau}^{\epsilon})\geqslant \eta \right)
\leqslant \mathbb{P}(\xi>c)+\exp\left\{-\sup_{\lambda \in \rr}[\lambda \eta-\delta H(\lambda c)]\right\}.\]
Since $\mathbb{P}(\xi>c)=\mathbb{P}(M_{T+1}^{\epsilon*}>{c}/{K}-1)$, it follows that
\[\mathbb{P}\left(\sup_{t\leqslant \delta}(M_{\tau+t}^{\epsilon}-M_{\tau}^{\epsilon})\geqslant \eta \right)
\leqslant 2 \max \left(\mathbb{P}\left(M_{T+1}^{\epsilon*}>\frac{c}{K}-1\right),\exp\left\{-\sup_{\lambda \in \rr}[\lambda \eta-\delta H(\lambda c)]\right\}\right).\]
The supremum of $\lambda \eta-\delta H(\lambda c)$ can be explicitly evaluated:
\[\sup_{\lambda \in \rr}[\lambda \eta-\delta H(\lambda c)]=\sup_{\lambda \in \rr}\left[\lambda \eta-\delta|\lambda|c-\delta\frac{\lambda^2 c^2 \epsilon}{2}\right]
=\frac{1}{\epsilon}\sup_{\lambda >0}\left[(\eta \epsilon-\delta c\epsilon)\lambda-\frac{\delta c^2 \epsilon^2}{2}\lambda^2\right]=\frac{(\eta-\delta c)^2}{2\epsilon\delta c^2}.\]
As a consequence, for all positive $c$,
 \[\lim_{\delta \rightarrow 0} \limsup_{\epsilon \rightarrow 0} \epsilon \log \exp\left\{-\sup_{\lambda \in \rr}[\lambda \eta-\delta H(\lambda c)]\right\}=
 \lim_{\delta \rightarrow 0}-\frac{(\eta-\delta c)^2}{2\delta c^2}=-\infty.\]
It is concluded that for any $\tau \in \Gamma_{T}(\mathcal{F}_t)$ and $c>0$,
\[\lim_{\delta \rightarrow 0} \limsup_{\epsilon \rightarrow 0} \epsilon \log \mathbb{P}\left(\sup_{t\leqslant \delta}(M_{\tau+t}^{\epsilon}-M_{\tau}^{\epsilon})\geqslant \eta \right)
\leqslant \limsup_{\epsilon \rightarrow 0} \epsilon \log  \mathbb{P}\left(M_{T+1}^{\epsilon*}>\frac{c}{K}-1\right).
\]
By (\ref{isup}), we know
\[\lim_{c \rightarrow \infty}\limsup_{\epsilon \rightarrow 0} \epsilon \log \mathbb{P}\left(M_{T+1}^{\epsilon*}>\frac{c}{K}-1\right)=-\infty.\]
It implies
\begin{eqnarray*}\lefteqn{\lim_{\delta \rightarrow 0} \limsup_{\epsilon \rightarrow 0} \epsilon \log \sup_{\tau \in \Gamma_{T}(\mathcal{F}_t)} \mathbb{P}\left(\sup_{t\leqslant \delta}(M_{\tau+t}^{\epsilon}-M_{\tau}^{\epsilon})
\geqslant \eta \right)}\\
&\leqslant& \inf_{c>0} \limsup_{\epsilon \rightarrow 0} \epsilon \log \mathbb{P}\left(M_{T+1}^{\epsilon*}>\frac{c}{K}-1\right) =-\infty.
\end{eqnarray*}
Moreover, the claim
\[\lim_{\delta \rightarrow 0} \limsup_{\epsilon \rightarrow 0} \epsilon \log \sup_{\tau \in \Gamma_{T}(\mathcal{F}_t)} \mathbb{P}\left(\sup_{t\leqslant \delta}(M_{\tau}^{\epsilon}-M_{\tau+t}^{\epsilon})\geqslant \eta \right)
 =-\infty
\]
is proved in the same way. Thus the desired claim follows from (\ref{sep}). 
 \end{proof}

\section{Auxiliary results}
We first identify a dense subset of $\mathbb{C}_T \times \mathbb{M}_T$ which substantially simplifies the proof of the local LDP in the next two sections.
Let $\mathbb{M}_T^+$ be the subset of $\mathbb{M}_T$ such that $K_{\nu}(s, i)>0, \forall s\in[0, T], i \in \mathbb{S}$, and let $\mathbb{M}_T^{++}$ be the subset of $\mathbb{M}_T^+$ such that 
$K_{\nu}(\cdot, i) \in \mathbb{C}_{[0, T]}^{\infty}, \forall i \in \mathbb{S}$.  
\begin{lemma} \label{dense}
$\mathbb{M}_T^{++}$ is dense in $\mathbb{M}_T$.
\end{lemma}
\begin{proof}
We prove the claim in two steps. Firstly, we show that $\mathbb{M}_T^{++}$ is dense in $\mathbb{M}_T^{+}$.
We begin by introducing the standard {\it mollifier} $J(x)$ on $\rr$, i.e.,
\begin{equation*}
J(x):= \left\{
 {\displaystyle 
  \begin{array}{l l}{\displaystyle 
k\exp\left(\frac{1}{|x|^2-1}\right) }& \quad 
 \text{ if } |x|<1,\\
   0 & \quad \text{ if } |x| \geqslant 1,
  \end{array} }\right.
\end{equation*}
where $k>0$ is selected so that $\int_{\rr}J(x)\DD x=1.$ For each $\eta>0$, we define $J_{\eta}(x):={\eta}^{-1}J({x}/{\eta})$.
For any $\nu$ in $\mathbb{M}_T^+$, we extend the domain of $K_{\nu}(\cdot, i)$ to $(-1, T+1)$, as follows: 
 \begin{equation*}
K_{\nu}(s, i):= \left\{
 {\displaystyle 
  \begin{array}{l l}{\displaystyle 
K_{\nu}(s, i) } & \quad  \text{ if } s \in [0, T],\\
K_{\nu}(0, i) & \quad \text{ if } s \in (-1, 0),\\
K_{\nu}(1, i) & \quad \text{ if } s \in (T, T+1).\end{array} }\right.
\end{equation*}
Since $K_{\nu}(\cdot, i)$ is integrable on $(-1, T+1)$ for each $i \in \mathbb{S}$, we can define its mollification as
\[K_{\nu}^{\eta}(s, i)=\int_{-\eta}^{\eta}J_{\eta}(y)K_{\nu}(s-y, i)\DD y,  \text{ for } s\in (-1+\eta, T+1-\eta), \eta <1.\]
By Theorem C.6 in Evans \cite{evans},  $K_{\nu}^{\eta}(\cdot, i)$ is smooth on $(-1+\eta, T+1-\eta)$ and $K_{\nu}^{\eta}(\cdot , i) \rightarrow K_{\nu}(\cdot , i)$ almost everywhere as $\eta \rightarrow 0$.

Next we proceed to show that $\nu^{\eta}$ with the kernel  $K_{\nu}^{\eta}(s, i)$ is an element of $\mathbb{M}_T^{++}$. It is clear that $K_{\nu}^{\eta}(\cdot, i)\in  \mathbb{C}_{[0,T]}^{\infty}$ when restricted to $[0, T]$, and in addition we have $K_{\nu}^{\eta}(s, i)>0,$ for all $s\in[0, T]$ and $i \in \mathbb{S}$. As a consequence, we only need to prove that $\sum_{i=1}^{d}K_{\nu}^{\eta}(s, i)=1$, for all $s \in [0, T]$.  For any $ s \in [0, T]$ and $y \in [-\eta, \eta]$, it holds that $\sum_{i=1}^{d}K_{\nu}(s-y, i)=1$ since  $s-y \in (-1, T+1)$. We thus have that
\[\sum_{i=1}^{d}K_{\nu}^{\eta}(s, i)=\int_{-\eta}^{\eta}J_{\eta}(y)\sum_{i=1}^{d}K_{\nu}(s-y, i)\DD y=\int_{-\eta}^{\eta}J_{\eta}(y)\DD y=1.\]

For any $t \leqslant T, i \in \mathbb{S}$, due to the fact that $K_{\nu}^{\eta}(\cdot , i) \rightarrow K_{\nu}(\cdot , i)$ on $(0, t)$ almost everywhere  as $\eta \rightarrow 0$ and $K_{\nu}^{\eta}(s , i) \leqslant 1$, it holds that 
\[\int_{0}^{t}K_{\nu}^{\eta}(s , i) \DD s \rightarrow \int_{0}^{t}K_{\nu}(s , i) \DD s, \text{ as } \eta \rightarrow 0,\] appealing to the dominated convergence theorem. Since $\int_{0}^{t}K_{\nu}^{\eta}(s , i) \DD s$ is increasing in $t$, $\int_{0}^{t}K_{\nu}(s , i) \DD s$ is continuous in $t$ and $d$ is finite, we obtain the following uniform convergence:
\[d_T(\nu^{\eta}, \nu)=\sup_{t\in [0, T], i \in \mathbb{S}}\left|\int_{0}^{t}K_{\nu}^{\eta}(s , i) \DD s-\int_{0}^{t}K_{\nu}(s , i) \DD s\right|\rightarrow 0, \text{ as } \eta \rightarrow 0, \]
using Result 1.1.21 in Jacod \cite{jacod1}. 

Secondly, we prove that $\mathbb{M}_T^{+}$ is dense in $\mathbb{M}_T$. Noticing that $K_{\nu}(s, i)$ can be 0 for some $ i, s$, we define (for any $\nu \in  \mathbb{M}$) an $\nu^{\eta} \in \mathbb{M}_T^+$ through
\[K_{\nu}^{\eta}(s, i):=\frac{K_{\nu}(s, i)+\eta}{1+\eta d},\] $\eta>0$, for all $ i, s$. As is directly verified,  $d_T(\nu^{\eta}, \nu)\leqslant T({\eta+\eta d})/({1+\eta d}).$ Then the desired result holds.
Consequently, $\mathbb{M}_T^{++}$ is dense in $\mathbb{M}_T$ by the triangle inequality.
\end{proof}

We then present a regularity property of the rate function $\tilde{I}_T(\nu)$ on $\mathbb{M}_T^{++}$.
\begin{lemma}\label{optimizer} Fix $s\in[0,T]$ and $\nu \in \mathbb{M}_T^{++}$. 
Then there is an optimizer $u^*(s,  \cdot)$ of \[\inf_{u \in U}\left[\sum_{i=1}^{d}\frac{(Qu)(i)}{u(i)}K_{\nu}(s, i)\right]\]  such that
$u^*(\cdot,  i) \in \mathbb{C}_{[0, T]}^{\infty},$ for all $i \in \mathbb{S}$, and $u^* \in \mathbb{U}$.
\end{lemma}
\begin{proof}
As obviously
\[\sum_{i=1}^{d}\frac{(Qu)(i)}{u(i)} K_{\nu}(s, i)=\sum_{i=1}^{d}Q_{ii} K_{\nu}(s, i)+\sum_{i=1}^{d}\frac{\sum_{j\neq i}^{d}Q_{ij}u(j)}{u(i)} K_{\nu}(s, i),\] the optimization problem essentially reduces to
\[\inf_{u \in U}\left[\sum_{i=1}^{d}\frac{\sum_{j\neq i}^{d}Q_{ij}u(j)}{u(i)} K_{\nu}(s, i)\right].\] We let $r_{ji}:={u(j)}/{u(i)}$, for $i \neq j$. Since $r_{ji}=1/r_{ij}$, the optimization problem can be written as a minimization over $d(d-1)/2$ variables:
\[\inf_{r_{ji}>0}\sum_{i=1}^d \sum_{j=1}^{i-1}  \left[ Q_{ij}r_{ji}K_{\nu}(s, i)+Q_{ji}r_{ji}^{-1}K_{\nu}(s, j). \right]\]
Observe that for any $i,j,k$ the equality $r_{ij} r_{jk}= r_{ik}$ needs to hold, which corresponds to $\psi(d):=(d-1)(d-2)/2$ constraints. 
We then perform the change of variables $x_{ji}:=\log r_{ji}$,  and denote by ${\boldsymbol X}=(x_{21}, \cdots, x_{d(d-1)})^{\rm T}$ the $d(d-1)/2$ variables.
Letting ${\boldsymbol K}_{\nu}(s)=(K_{\nu}(s, 1), \cdots, K_{\nu}(s, d))^{\rm T}$, we transform the above optimization problem into
\[\inf_{\boldsymbol{X}} f(\boldsymbol{K}_{\nu}(s), \boldsymbol{X}), \:\:\text{ where } f(\boldsymbol{K}_{\nu}(s), \boldsymbol{X}):=\sum_{i=1}^d \sum_{j=1}^{i-1}  \left[ Q_{ij}e^{x_{ji}}K_{\nu}(s, i)+Q_{ji}e^{-x_{ji}}K_{\nu}(s, j) \right], \]
with $(d-1)(d-2)/2$ additional constraints to be imposed.

The gradient vector of $f$ with respect to $\boldsymbol{X}$ is
\[\bold{D}_{ \boldsymbol{X}}f=\left(\frac{\partial f}{\partial x_{21}}, \cdots, \frac{\partial f}{\partial x_{d(d-1)}} \right),\]
and the corresponding Hessian matrix 
$\boldsymbol{D}_{\boldsymbol{X}}^2 f $ is the diagonal matrix which has entries of the form $Q_{ij}e^{x_{ji}}K_{\nu}(s, i)+Q_{ji}e^{-x_{ji}}K_{\nu}(s, j)$ on its diagonal.
The idea is now to split the vector $\boldsymbol{X}$ into $\boldsymbol{X}_0=(x_{21}, \cdots, x_{d1})^{\rm T}$  and  $\boldsymbol{X}_1$ (where the latter vector corresponds with the remaining $\psi(d)$ variables). Due to the constraints, we have $\boldsymbol{X}_1= \boldsymbol{L} \boldsymbol{X}_0$ where $\boldsymbol{L}$ is a matrix of dimension $\psi(d) \times (d-1)$.  
The next step is to include the constraints  into the optimization equation $f$.  It yields the following new optimization problem, on which no additional constraints need to ne imposed anymore:
\[\inf_{\boldsymbol{X}_0}  \hat{f}(\boldsymbol{K}_{\nu}(s), \boldsymbol{X}_0), \text{ where } \hat{f}(\boldsymbol{K}_{\nu}(s), \boldsymbol{X}_0)=f(\boldsymbol{K}_{\nu}(s), (\boldsymbol{X}_0, \boldsymbol{L} \boldsymbol{X}_0)).\]
{Observe that $f$ is a globally strictly convex function of $\boldsymbol{X}$ on a convex domain, and consequently $\hat{f}$ is a strictly convex function of $\bold{X}_0$.
Hence, there is a unique minimizer $\boldsymbol{X}^*_0(s)=(x^*_{21}(s), \cdots, x^*_{d1}(s))^{\rm T}$ for any $s \in [0, T]$. Since we have that both $K_{\nu}(s, i)>0$ and $Q_{ij}>0$ for $ i \neq j$, any entry of $\boldsymbol{X}^*_0(s)$ cannot be $-\infty$ or $\infty$. We thus conclude that $\boldsymbol{X}^*_0(s) \in \rr^{d-1}$.}

Let $\boldsymbol{I}$ be the $(d-1)$-dimensional identity matrix. Then we define 
\[\tilde{f}(\boldsymbol{K}_{\nu}(s), \boldsymbol{X}_0):=\bold{D}_{\boldsymbol{X}_0}\hat{f}(\boldsymbol{K}_{\nu}(s), \boldsymbol{X}_0)=\bold{D}_{\boldsymbol{X}}f(\boldsymbol{K}_{\nu}(s), (\boldsymbol{X}_0, \boldsymbol{L} \boldsymbol{X}_0)){\boldsymbol{I} \choose \boldsymbol{L}},\]
which is a smooth function on $ \rr^d\times \rr^{d-1}$ such that $\tilde{f}(\boldsymbol{K}_{\nu}(s), \boldsymbol{X}^*_0(s) )=0$.
The gradient matrix of $\tilde{f}$ with respect to $ \boldsymbol{X}_0$ evaluated in $\boldsymbol{X}^*_0(s)$ is
\[\boldsymbol{G}:=\bold{D}_{\boldsymbol{X}_0}^2\tilde{f}(\boldsymbol{K}_{\nu}(s), \boldsymbol{X}^*_0(s))=(\boldsymbol{I} \:\:\: \boldsymbol{L}^{\rm T})\, \bold{D}_{\boldsymbol{X}}^2 f (\boldsymbol{K}_{\nu}(s), (\boldsymbol{X}^*_0(s),\boldsymbol{L} \boldsymbol{X}^*_0(s))){\boldsymbol{I} \choose \boldsymbol{L}}.\]
Let $|\boldsymbol{G}|$ denote the determinant of $\boldsymbol{G}$. Since $\boldsymbol{H}$ is a positive-definite diagonal matrix and $\boldsymbol{L}$ is of full rank, we conclude that $|\boldsymbol{G}|\neq 0,$ for all $s \in (0, T)$.

Hence, the implicit function theorem (cf.\ Theorem C.8 in Evans \cite{evans}) implies that $\boldsymbol{X}^*_0(s)$ is a smooth function of $\boldsymbol{K}_{\nu}(s)$: since $K_{\nu}(\cdot, i) \in \mathbb{C}_{[0, T]}^{\infty}$ for all $i \in \mathbb{S}$, we conclude that $\boldsymbol{X}^*_0(s) \in \mathbb{C}_{[0, T]}^{\infty}(\rr^{d-1})$. It also follows that the corresponding minimizer in terms of the variables $r_{ij}$, say $(r^*_{21}(s), \cdots, r^*_{d 1} (s))$, is in $\mathbb{C}_{[0, T]}^{\infty}((0, \infty)^{d-1})$. Recalling that $r_{ji}={u(j)}/{u(i)}$, we set 
\[(u^*(s, 1), u^*(s, 2), \cdots, u^*(s, d))\equiv (1, r^*_{21}(s), \cdots, r^*_{d 1} (s) )\] on $[0, T]$. Then $(u^*(s, 1), \cdots, u^*(s, d))$ is an optimizer corresponding to
\[\inf_{u \in U}\left[\sum_{i=1}^{d}\frac{(Qu)(i)}{u(i)}K_{\nu}(s, i)\right],\] and $u^*(\cdot,  i) \in \mathbb{C}_{[0, T]}^{\infty}$ for all $i \in \mathbb{S}$.  
It is easily seen that $u^*(s, i)> 0$. Then $\inf_{s \in [0, T], i \in \mathbb{S}} u^*(s, i)>0$ by continuity of $u^*$ on $[0, T]$. Hence, $u^* \in \mathbb{U}$.
\end{proof}

The following continuity property of the rate functions will be used in proving the upper and lower bounds. 
\begin{lemma} \label{continuity}
Let $\nu^{\eta}, \nu \in \mathbb{M}_T$ with kernels $K_{\nu}^{\eta}$ and $K_{\nu}$ such that $K_{\nu}^{\eta}(\cdot, i) \rightarrow K_{\nu}(\cdot, i)$ a.e.\ as $\eta \rightarrow 0$ on $[0, T]$  for each $i\in \mathbb{S}$. Then
\begin{itemize}
\item[(i)] $\tilde{I}_{T}(\nu^{\eta}) \rightarrow \tilde{I}_{T}(\nu)$ as $\eta \rightarrow 0$;
\item[(ii)] $I_T(\varphi, \nu^{\eta}) \rightarrow I_T(\varphi, \nu)$ as $\eta \rightarrow 0$, $ \forall \varphi \in \mathbb{H}_T$, if \,$\inf_{i, x} \sigma^{2}(i, x)>0$.\end{itemize}
\end{lemma}
\begin{proof}
(i) Let $\rho$ be a $d$-dimensional vector such that $\sum_{i=1}^{d}\rho(i)=1$ and $\rho(i) \geqslant 0$.
By Lemma 4.22 in den Hollander \cite{Den hollander},  \[- \inf_{u \in U}\left[\sum_{i=1}^{d}\frac{(Qu)(i)}{u(i)} \rho(i)\right]\] is continuous in $\rho$ and positive. Moreover, for all $\rho$,  realizing that the $Q_{ii}$ are negative,
\[ -\inf_{u \in U}\left[\sum_{i=1}^{d}\frac{(Qu)(i)}{u(i)} \rho(i)\right] \leqslant -\sum_{i=1}^{d}Q_{ii} \] as a consequence of
\[ -\inf_{u \in U}\left[\sum_{i=1}^{d}\frac{(Qu)(i)}{u(i)} \rho(i)\right]=-\sum_{i=1}^{d}Q_{ii} \rho(i)-\inf_{u \in U}\left[\sum_{i=1}^{d}\frac{\sum_{j\neq i}^{d}Q_{ij}u(j)}{u(i)} \rho(i)\right].\]
Hence, \[\sup_{u \in U}\left[-\sum_{i=1}^{d}\frac{(Qu)(i)}{u(i)}K_{\nu}^{\eta}(s, i)\right] \rightarrow \sup_{u \in U}\left[-\sum_{i=1}^{d}\frac{(Qu)(i)}{u(i)}K_{\nu}(s, i)\right],\] as $\eta \rightarrow 0$, almost everywhere on $[0, T]$. Also, \[-\inf_{u \in U}\left[\sum_{i=1}^{d}\frac{(Qu)(i)}{u(i)}K_{\nu}^{\eta}(s, i)\right] \leqslant -\sum_{i=1}^{d}Q_{ii}\] for all $s$.
Then the desired result follows directly by applying the dominated convergence theorem. 

(ii) When $\inf_{i, x} \sigma^{2}(i, x)>0$, it is easily seen by continuity that
\[\frac{\varphi'_{t}-\hat{b}(\nu^{\eta}, \varphi_{t})]^2}{\hat{\sigma}^2(\nu^{\eta}, \varphi_{t}) }\rightarrow \frac{[\varphi'_{t}-\hat{b}(\nu, \varphi_{t})]^2}{\hat{\sigma}^2(\nu, \varphi_{t})}\] a.e.\ as $\eta \rightarrow 0$.
Let $\sigma^2$ denote $\inf_{i, x} \sigma^{2}(i, x)$. For every $\nu \in \mathbb{M}_T$, we have
\[ \frac{[\varphi'_{t}-\hat{b}(\nu, \varphi_{t})]^2}{\hat{\sigma}^2(\nu, \varphi_{t})} \leqslant \frac{|\varphi'_{t}|^2+(\sum_{i=1}^{d}|b(i, \varphi_t)|)^2+2|\varphi'_{t}|(\sum_{i=1}^{d}|b(i, \varphi_t)|)}{\sigma^2}.\]
Since $\varphi$ is absolutely continuous and $b(i, x)$ is Lipschitz continuous in $x$, $\sum_{i=1}^{d}|b(i, \varphi_t)|< b < \infty$ on $[0, T]$. Hence, 
$ [\varphi'_{t}-\hat{b}(\nu, \varphi_{t})]^2/\hat{\sigma}^2(\nu, \varphi_{t})\leqslant (|\varphi'_{t}|+b)^2/\sigma^2$. Since $\varphi'$ is square-integrable on $[0, T]$,  $I_T(\varphi, \nu^{\eta}) \rightarrow I_T(\varphi, \nu)$ as $\eta \rightarrow 0$ by again applying the dominated convergence theorem.
\end{proof}

\section{Upper bound for the local LDP}
This section considers the upper bound in the local LDP, whereas the next section concentrates on the corresponding lower bound. 
Recall that our  aim is to establish
\[\limsup_{\delta\rightarrow 0}\limsup_{\epsilon \rightarrow 0} \epsilon \log \PP
(\rho_T(M^{\epsilon}, \varphi)+d_T(\nu^{\epsilon}, \nu)\leqslant \delta)\leqslant -L_{T}(\varphi, \nu),\]
with $L_T(\varphi,\nu)$ as defined in Section \ref{prem}.
Our approach, which has a simlar structure as the one used in Liptser in \cite{Liptser1}, finds an exponential (in $\epsilon$, that is) upper bound on 
the probability $\PP
(\rho_T(M^{\epsilon}, \varphi)+d_T(\nu^{\epsilon}, \nu)\leqslant \delta)$ relying on the method of 
stochastic exponentials. As it turns out, this bound should contain the rate function $L_{T}(\varphi, \nu)$, as desired.

We start by introducing some additional notation.
Let $\mathbb{S}_T$ denote the space of all step functions on $[0, T]$ of the form,
for $k\in{\mathbb N}$ and real numbers $\lambda_0,\cdots,\lambda_k$,
\[\lambda(t)=\lambda_{0}\one_{\{t=0\}}(t)+\sum_{i=0}^{k}\lambda_{i}\one_{(t_{i}, t_{i+1}]}(t), 
\:\:\:\:0=t_0<\cdots<t_{k+1}=T.\]
For any $\varphi \in \mathbb{C}_{T}$, we introduce the following notation
\[
\int_{0}^{T}\lambda(s)\DD \varphi_{s}:=\sum_{i=0}^{k}\lambda_{i}[\varphi_{T\wedge t_{i+1}}-\varphi_{T\wedge t_{i}}].
\]
In the sequel we frequently use the process
\[N_{t}^{\epsilon}:=\frac{1}{\sqrt{\epsilon}}\int_{0}^{t}\lambda(s)\sigma(X_{s}^{\epsilon}, M_{s}^{\epsilon})\DD B_{s}
,\:\:\:\:\: \lambda \in \mathbb{S}_T,\]
which has the stochastic exponential 
\[\eee(N^{\epsilon})_{t}=\exp\left(N_{t}^{\epsilon}-\frac{1}{2}\langle N^{\epsilon}\rangle_{t}\right),\:\:\mbox{where}\:\:\:\:
 \langle N^{\epsilon}\rangle_{t}=\frac{1}{\epsilon}\int_{0}^{t}\lambda^2(s)\sigma^2(X_{s}^{\epsilon}, M_{s}^{\epsilon})\DD s.\]

Next we introduce a stochastic exponential associated with the occupation measure $\nu^{\epsilon}$. For any $u( \cdot, \cdot) \in \mathbb{U}$,   
\[\hat{N}_t^{\epsilon}=u(t, X_t^{\epsilon})-u(0, X^{\epsilon}_0)-\int_{0}^{t} \frac{\partial}{\partial s} u(s, X^{\epsilon}_s)\DD s-\int_0^t(Q^{\epsilon}u)(s, X^{\epsilon}_s) \DD s\]
is a local martingale on $[0, T]$ due to It\^{o}'s formula.  We define
\[\tilde{N}_{t}^{\epsilon}:=\int_{0}^{t} \frac{1}{u(s-, X^{\epsilon}_{s-})}\DD \hat{N}_s^{\epsilon}.\]
Then
\begin{equation}\label{N}
\eee(\tilde{N}^{\epsilon})_t=\frac{u(t, X^{\epsilon}_{t})}{u(0, X^{\epsilon}_{0})}\exp\left(-\int_{0}^{t}\frac{\frac{\partial}{\partial s} u(s, X^{\epsilon}_s)+(Q^{\epsilon}u)
(s, X^{\epsilon}_{s})}{u(s, X_{s}^{\epsilon})}\DD s\right)\end{equation}
is the stochastic exponential of $\tilde{N}_{t}^{\epsilon}$.  Indeed,
\begin{eqnarray*}
\DD \eee(\tilde{N}^{\epsilon})_t
&=& \frac{u(t, X^{\epsilon}_{t})}{u(0, X^{\epsilon}_{0})}\exp\left(-\int_{0}^{t}\frac{\frac{\partial}{\partial s} u(s, X^{\epsilon}_s)+(Q^{\epsilon}u)
(s, X^{\epsilon}_{s})}{u(s, X^{\epsilon}_{s})}\DD s\right)\\
&&\hspace{2cm}\times
\left(-\frac{\frac{\partial}{\partial t} u(t, X^{\epsilon}_t)+(Q^{\epsilon}u)
(t, X^{\epsilon}_{t})}{u(t, X^{\epsilon}_{t})}\DD t\right)\\
&&\hspace{2cm}+\:\exp\left(-\int_{0}^{t}\frac{\frac{\partial}{\partial s} u(s, X^{\epsilon}_s)+(Q^{\epsilon}u)
(s, X^{\epsilon}_{s})}{u(s, X^{\epsilon}_{s})}\DD s\right)\frac{\DD u(t, X^{\epsilon}_{t})}{u(0, X^{\epsilon}_{0})}\\
&=& \frac{\eee(\tilde{N}^{\epsilon})_{t-}}{u(t-, X^{\epsilon}_{t-})}\left[\DD u(t, X^{\epsilon}_{t})-\frac{\partial}{\partial t} u(t, X^{\epsilon}_t)\DD t-(Q^{\epsilon}u)
(t, X_{t})\DD t\right]\\
&=& \frac{\eee(\tilde{N}^{\epsilon})_{t-}}{u(t-, X^{\epsilon}_{t-})}\DD \hat{N}_t^{\epsilon}.
\end{eqnarray*}
Since $\inf_{t \in [0,T], i \in \mathbb{S}} u(t, i)>0$, $\tilde{N}_{t}^{\epsilon}$ is a local martingale and its stochastic exponential $\eee(\tilde{N}^{\epsilon})_t$ is also a local martingale by Theorem 1.4.61 in Jacod and Shiryaev \cite{jacod}. Then
$\eee(\tilde{N}^{\epsilon})_t$ is a martingale since it is bounded. We will use this martingale property when applying a change of measure in the next section.  The martingale $\eee(\tilde{N}^{\epsilon})_t$ is an extension of the exponential martingale studied by Palmowski and Rolski in \cite{Palmowski}.
\begin{lemma}\label{supermartingale}
$ \eee(\tilde{N}^{\epsilon})_t\eee(N^{\epsilon})_t$ is a local martingale, and 
$\ee[\eee(\tilde{N}^{\epsilon})_t\eee(N^{\epsilon})_t]\leqslant 1$. 
\end{lemma}

\begin{proof}
By Protter \cite[Thm.\ 2.38]{Protter},
\[\eee(\tilde{N}^{\epsilon})_t\eee(N^{\epsilon})_t=\eee(\tilde{N}^{\epsilon}+N^{\epsilon}+[\tilde{N}^{\epsilon}, N^{\epsilon}])_t,\]
where $[\tilde{N}^{\epsilon}, N^{\epsilon}]$ denotes the quadratic covariation process. Since $\tilde{N}_{t}^{\epsilon}$ is a pure jump
 local martingale and $N_{t}^{\epsilon}$ is a continuous local martingale,
$[\tilde{N}^{\epsilon}, N^{\epsilon}]=0$. Then $\eee(\tilde{N}^{\epsilon})_t\eee(N^{\epsilon})_t$ is the stochastic exponential of
the local martingale $\tilde{N}^{\epsilon}_t+N^{\epsilon}_t$ and a local martingale too. Since a positive local martingale is a supermartingale,
$\ee[\eee(\tilde{N}^{\epsilon})_t\eee(N^{\epsilon})_t]\leqslant \ee[\eee(\tilde{N}^{\epsilon})_0 \eee(N^{\epsilon})_0]=1.$
\end{proof}

The above lemma evidently implies that
\[\ee\left[\one_{\{\rho_T(M^{\epsilon, T}, \varphi)
+d_T(\nu^{\epsilon}, \nu)\leqslant \delta\}}\eee(\tilde{N}^{\epsilon})_T \eee(N^{\epsilon})_T\right]
\leqslant 1.\] In order to find an exponential upper bound on
$\PP(\rho_T(M^{\epsilon}, \varphi)+d_T(\nu^{\epsilon}, \nu)\leqslant \delta)$, 
we derive non-random exponential \emph{lower} bounds on $\eee(\tilde{N}^{\epsilon})_T$ and
$\eee(N^{\epsilon})_T$ in case that both $M^{\epsilon}$ is close to $\varphi$ and $\nu^{\epsilon}$ close to $\nu$
(i.e., 
on the set $\{\rho_T(M^{\epsilon}, \varphi)+d_T(\nu^{\epsilon}, \nu)\leqslant \delta\}$). The next two lemmas present the results;
Lemma \ref{nonrandom bound 1} focuses on $\eee(N^{\epsilon})_T$, whereas
Lemma \ref{nonrandom bound 2} covers $\eee(\tilde N^{\epsilon})_T$.

\begin{lemma}\label{nonrandom bound 1}
For every $(\varphi, \nu) \in \mathbb{C}_T \times \mathbb{M}_T$ and every $\lambda \in \mathbb{S}_T$, $\delta>0$,  there exists a positive constant
$K_{\lambda, \varphi, T}$ not depending on $\epsilon$ or $\delta$ such that
\[\eee(N^{\epsilon})_T \geqslant  \exp\left\{\frac{1}{\epsilon}\left[\int_{0}^{T}\lambda(s)\DD \varphi_{s}
-\int_{0}^{T}\lambda(s)\hat{b}(\nu, \varphi_{s})\DD s
-\int_{0}^{T}\frac{\lambda^2(s)}{2}\hat{\sigma}^2(\nu, \varphi_{s})\DD s\right]
-\frac{\delta}{\epsilon} K_{\lambda, \varphi, T} \right\}\]
on the set
$\{\rho_{T}(M^{\epsilon}, \varphi)+d_T(\nu^{\epsilon}, \nu)\leqslant \delta\}.$
\end{lemma}
\begin{proof}
 It is first realized that, by (\ref{def}), $N_{t}^{\epsilon}$ can be rearranged as 
\[N_{t}^{\epsilon}=\frac{1}{\epsilon}\left[\int_{0}^{t}\lambda(s)\DD M_{s}^{\epsilon}
-\int_{0}^{t}\lambda(s)b(X_{s}^{\epsilon}, M_{s}^{\epsilon})\DD s\right].\]
Then a straightforward computation yields that
\begin{eqnarray*}
\lefteqn{N_{T}^{\epsilon}-\frac{1}{2}\langle N^{\epsilon}\rangle_{T}
=\frac{1}{\epsilon}\left[\int_{0}^{T}\lambda(s)\DD M_{s}^{\epsilon}
-\int_{0}^{T}\lambda(s)b(X_{s}^{\epsilon}, M_{s}^{\epsilon})\DD s
-\int_{0}^{T}\frac{\lambda^2(s)}{2}\sigma^2(X_{s}^{\epsilon}, M_{s}^{\epsilon})\DD s\right]}\\
&=& \frac{1}{\epsilon}\left[
\int_{0}^{T}\lambda(s)\DD M_{s}^{\epsilon}-\int_{0}^{T}\lambda(s)\DD\varphi_{s}\right]
-\frac{1}{\epsilon} \left[\int_{0}^{T}\lambda(s)b(X_{s}^{\epsilon}, M_{s}^{\epsilon})\DD s
-\int_{0}^{T}\lambda(s)\hat{b}(\nu, \varphi_{s})\DD s\right]\\
& &-\frac{1}{\epsilon} \left[\int_{0}^{T}\frac{\lambda^2(s)}{2}\sigma^2(X_{s}^{\epsilon}, M_{s}^{\epsilon})\DD s
-\int_{0}^{T}\frac{\lambda^2(s)}{2}\hat{\sigma}^2(\nu, \varphi_{s})\DD s\right]\\
& &+\frac{1}{\epsilon}\left[\int_{0}^{T}\lambda(s)\DD \varphi_{s}
-\int_{0}^{T}\lambda(s)\hat{b}(\nu, \varphi_{s})\DD s
-\int_{0}^{T}\frac{\lambda^2(s)}{2}\hat{\sigma}^2(\nu, \varphi_{s})\DD s\right].
\end{eqnarray*}
As a consequence, we evidently have
\begin{eqnarray*}
\lefteqn{N_{T}^{\epsilon}-\frac{1}{2}\langle N^{\epsilon}\rangle_{T}
\geqslant \frac{1}{\epsilon}\left[\int_{0}^{T}\lambda(s)\DD \varphi_{s}
-\int_{0}^{T}\lambda(s)\hat{b}(\nu, \varphi_{s})\DD s
-\int_{0}^{T}\frac{\lambda^2(s)}{2}\hat{\sigma}^2(\nu, \varphi_{s})\DD s\right]}\\
& &
-\frac{1}{\epsilon}\left|\int_{0}^{T}\lambda(s)\DD M_{s}^{\epsilon}-\int_{0}^{T}\lambda(s)\DD\varphi_{s}\right|
 -\frac{1}{\epsilon}\left|\int_{0}^{T}\lambda(s)b(X_{s}^{\epsilon}, M_{s}^{\epsilon})\DD s
 -\int_{0}^{T}\lambda(s)\hat{b}(\nu, \varphi_{s})\DD s\right|\\
& &
 -\frac{1}{\epsilon}\left|\int_{0}^{T}\frac{\lambda^2(s)}{2}\sigma^2(X_{s}^{\epsilon}, M_{s}^{\epsilon})\DD s
-\int_{0}^{T}\frac{\lambda^2(s)}{2}\hat{\sigma}^2(\nu, \varphi_{s})\DD s\right| \:\:\:\:\: \text{a.s..}
\end{eqnarray*}
Hence, by repeated use of the triangle inequality, we find that $N_{T}^{\epsilon}-\frac{1}{2}\langle N^{\epsilon}\rangle_{T}\geqslant
\epsilon^{-1} G_T^\epsilon$ a.s., where $G_T^\epsilon$ is given by
\begin{eqnarray*}
\lefteqn{\left[\int_{0}^{T}\lambda(s)\DD \varphi_{s}
-\int_{0}^{T}\lambda(s)\hat{b}(\nu, \varphi_{s})\DD s
-\int_{0}^{T}\frac{\lambda^2(s)}{2}\hat{\sigma}^2(\nu, \varphi_{s})\DD s\right]}\\
& &-\left|\int_{0}^{T}\lambda(s)\DD M_{s}^{\epsilon}-\int_{0}^{T}\lambda(s)\DD\varphi_{s}\right|
\\
&&-\left|\int_{0}^{T}\lambda(s)b(X_{s}^{\epsilon}, M_{s}^{\epsilon})-
\lambda(s)b(X_{s}^{\epsilon}, \varphi_{s})\DD s\right|
-\left|\int_{0}^{T}\lambda(s)b(X_{s}^{\epsilon}, \varphi_{s})-\lambda(s)\hat{b}(\nu, \varphi_{s})\DD s\right|\\
&&-\left|\int_{0}^{T}\frac{\lambda^2(s)}{2}\sigma^2(X_{s}^{\epsilon}, M_{s}^{\epsilon})-\frac{\lambda^2(s)}{2}\sigma^2(X_{s}^{\epsilon}, \varphi_{s})\DD s\right|
-\left|\int_{0}^{T}\frac{\lambda^2(s)}{2}\sigma^2(X_{s}^{\epsilon}, \varphi_{s})-\frac{\lambda^2(s)}{2}\hat{\sigma}^2(\nu, \varphi_{s})\DD s\right|\hspace{-0.7mm}.
\end{eqnarray*}

In the rest of the proof, all objects are considered on the set 
$\{\rho_{T}(M^{\epsilon}, \varphi)+d_T(\nu^{\epsilon}, \nu)\leqslant \delta\}$; we analyze all absolute values in the previous display separately. 
Let us start with considering the first absolute value; we thus find that
\[\left|\int_{0}^{T}\lambda(s)\DD M_{s}^{\epsilon}-\int_{0}^{T}\lambda(s)\DD\varphi_{s}\right|
=\left|\sum_{j=1}^{k}\lambda_{j}\left(M_{T\wedge t_{j+1}}^{\epsilon}-\varphi_{T \wedge t_{j+1}}-(M_{T\wedge t_{j}}^{\epsilon}-\varphi_{T \wedge t_{j}})\right)\right|
\leqslant 2 \lambda_{T}^{*}\delta.\]
Now consider the second absolute value. The Lipschitz condition  (A.1) implies that
\[\left|\int_{0}^{T}\lambda(s)b(X_{s}^{\epsilon}, M_{s}^{\epsilon})-
\lambda(s)b(X_{s}^{\epsilon}, \varphi_{s})\DD s\right|
\leqslant \int_{0}^{T}|\lambda(s)| K |M_{s}^{\epsilon}-\varphi_{s}|\DD s
\leqslant \lambda_{T}^{*} \delta KT. \]

For the fourth one, (A.1) also entails that
\[
\left|\int_{0}^{T}\frac{\lambda^2(s)}{2}\left[\sigma^2(X_{s}^{\epsilon}, M_{s}^{\epsilon})-\sigma^2(X_{s}^{\epsilon}, \varphi_{s})\right]\DD s\right|
\leqslant \frac{\lambda^{2*}_{T}}{2}\int_{0}^{T} K \delta
|\sigma(X_{s}^{\epsilon}, M_{s}^{\epsilon})+\sigma(X_{s}^{\epsilon}, \varphi_{s})|\DD s\]
Since $\varphi$ is continuous on $[0, T]$, there exists a positive constant $r$ such that $\varphi_{T}^{*} \leqslant r-\delta$.
It yields that $M^{\epsilon *}_{T}\leqslant r$ on the set 
$\{\rho_{T}(M^{\epsilon}, \varphi)+d_T(\nu^{\epsilon}, \nu)\leqslant \delta\}$. By the linear growth condition (A.2) and the above reasoning 
\[
|\sigma(X_{s}^{\epsilon}, M_{s}^{\epsilon})+\sigma(X_{s}^{\epsilon}, \varphi_{s})|
\leqslant K(1+|M_{s}^{\epsilon}|)+K(1+|\varphi_{s}|)
\leqslant 2K(1+r).
\]
We conclude that
\[
\left|\int_{0}^{T}\frac{\lambda^2(s)}{2}\left[\sigma^2(X_{s}^{\epsilon}, M_{s}^{\epsilon})-\sigma^2(X_{s}^{\epsilon}, \varphi_{s})\right]\DD s\right|
\leqslant \lambda^{2*}_{T}  \delta K^2(1+r)T.
\]
Then, concerning the third absolute value, 
\begin{eqnarray*}
\left|\int_{0}^{T}\lambda(s)b(X_{s}^{\epsilon}, \varphi_{s})-\lambda(s)\hat{b}(\nu, \varphi_{s})\DD s\right|
&=& \left|\sum_{j=0}^{k}\int_{t_{j}}^{t_{j+1}}\lambda_{j}[b(X_{s}^{\epsilon}, \varphi_{s})-\hat{b}(\nu, \varphi_{s})]\DD s\right|\\
&\leqslant& \sum_{j=0}^{k}\left|\int_{t_{j}}^{t_{j+1}}\lambda_{j}\sum_{i=1}^{d}
b(i, \varphi_{s})[\one_{\{X^{\epsilon}_s=i\}}-K_{\nu}(s, i)]\DD s\right|\\
&=& \sum_{j=0}^{k}\left|\int_{t_{j}}^{t_{j+1}}\sum_{i=1}^{d}f_{j}(i, s)
[\one_{\{X^{\epsilon}_s=i\}}-K_{\nu}(s, i)]\DD s\right|\\
&\leqslant& \sum_{j=0}^{k}\sum_{i=1}^{d}\left|\int_{t_{j}}^{t_{j+1}}f_{j}(i, s)
[\one_{\{X^{\epsilon}_s=i\}}-K_{\nu}(s, i)]\DD s\right|
\end{eqnarray*}
where $f_{j}(i, s):=\lambda_{j}b(i, \varphi_{s})$.  Since $b(i, \cdot)$ is Lipschitz continuous and $\varphi_s$ is absolutely continuous, $f_{j}(i, s)$ is of bounded variation. Then, by Lemma \ref{integrationbyparts}, 
\[\sup_{i \in \mathbb{S}} \left|\int_{t_{j}}^{t_{j+1}}f_{j}(i, s)
[\one_{\{X^{\epsilon}_s=i\}}-K_{\nu}(s, i)]\DD s\right| \leqslant C_j\delta,\]
where $C_j$ is a constant.  We thus conclude that
\begin{equation*}
\left|\int_{0}^{T}\lambda(s)b(X_{s}^{\epsilon}, \varphi_{s})-\lambda(s)\hat{b}(\nu, \varphi_{s})\DD s\right|
\leqslant \sum_{j=0}^{k} C_j  \delta d \leqslant  C \delta.
\end{equation*}
A similar procedure yields for the last absolute value
\begin{equation*}
\left|\int_{0}^{T}\frac{\lambda^2(s)}{2}\sigma^2(X_{s}^{\epsilon}, \varphi_{s})-\frac{\lambda^2(s)}{2}\hat{\sigma}^2(\nu, \varphi_{s})\DD s\right|
\leqslant  C'  \delta.
\end{equation*}
Upon collecting these inequalities, we find
\[\eee(N^{\epsilon})_T
\geqslant \exp\left\{\frac{1}{\epsilon}\left[\int_{0}^{T}\lambda(s)\DD \varphi_{s}
-\int_{0}^{T}\lambda(s)\hat{b}(\nu, \varphi_{s})\DD s
-\int_{0}^{T}\frac{\lambda^2(s)}{2}\hat{\sigma}^2(\nu, \varphi_{s})\DD s\right]-\frac{\delta}{\epsilon}K_{\lambda, \varphi, T}\right\} ,
\]
where we denote 
\[K_{\lambda, \varphi, T}:=2 \lambda_{T}^{*}+\lambda_{T}^{*} KT
+\lambda^{2*}_{T}  K^2(1+r)T
+ C + C'  , \]
which is a positive constant not depending on $\delta$ or $\epsilon$.
\end{proof}

\begin{lemma}\label{nonrandom bound 2}
For every $ \nu \in \mathbb{M}_T$, every $u \in \mathbb{U}$ and every $\gamma, \delta>0$,  there exist positive constants
$C_u$, $C'_u$, $K_u$ and $K_{Q, u}$ not depending on $\epsilon$ or $\delta$ such that
 \[\eee(\tilde{N}^{\epsilon})_T\geqslant K_u\exp\left(-C_u\delta d-\gamma d-C'_uTd-\frac{1}{\epsilon}K_{Q, u} \delta d-\frac{1}{\epsilon}
\int_{0}^{T} \sum_{i=1}^{d} \frac{Q u(s, i)}{u(s, i)} K_{\nu}(s, i)\DD s\right)\]
on the set 
$\{\rho_{T}(M^{\epsilon}, \varphi)+d_T(\nu^{\epsilon}, \nu)\leqslant \delta\}.$
\end{lemma}
\begin{proof} First observe that
\begin{eqnarray*}
\eee(\tilde{N}^{\epsilon})_T&=&
\frac{u(T, X^{\epsilon}_{T})}{u(0, X^{\epsilon}_{0})}\exp\left(-\int_{0}^{T}\sum_{i=1}^{d}\frac{\frac{\partial}{\partial s} u(s, i)+(Q^{\epsilon}u)
(s, i)}{u(s, i)} \one_{\{X_s^{\epsilon}=i\}}\DD s\right)\\
&=&
\frac{u(T, X^{\epsilon}_{T})}{u(0, X^{\epsilon}_{0})}\exp\left(-\sum_{i=1}^{d}
\int_{0}^{T}\frac{\frac{\partial}{\partial s} u(s, i)}{u(s, i)} [\one_{\{X_{s}^{\epsilon}=i\}}-K_{\nu}(s, i)]\DD s\right.\\
&&\hspace{4cm}\left.-\sum_{i=1}^{d}
\int_{0}^{T}\frac{\frac{\partial}{\partial s} u(s, i)}{u(s, i)} K_{\nu}(s, i)\DD s\right) \times\\
&& \,\ \exp\left(-\frac{1}{\epsilon}\sum_{i=1}^{d}
\int_{0}^{T}\frac{Q u(s, i)}{u(s, i)} [\one_{\{X_{s}^{\epsilon}=i\}}-K_{\nu}(s, i)]\DD s-\frac{1}{\epsilon}\sum_{i=1}^{d}
\int_{0}^{T}\frac{Q u(s, i)}{u(s, i)} K_{\nu}(s, i)\DD s\right)\end{eqnarray*}
By the definition of $u$ and $X_0^{\epsilon}=x$, we have that $K_u:=\min_{i, x} {u(T, i)}/{u(0, x)}$ is a positive constant.
Hence $\eee(\tilde{N}^{\epsilon})_T$ majorizes
\begin{eqnarray*}
\lefteqn{
K_u\exp\left(-\sum_{i=1}^{d}
\left|\int_{0}^{T}\frac{\frac{\partial}{\partial s} u(s, i)}{u(s, i)} [\one_{\{X_{s}^{\epsilon}=i\}}-K_{\nu}(s, i)]\DD s\right|-\sum_{i=1}^{d}
\left|\int_{0}^{T}\frac{\frac{\partial}{\partial s} u(s, i)}{u(s, i)} K_{\nu}(s, i)\DD s\right|\right) \times}\\
&& \,\ \exp\left(-\frac{1}{\epsilon}\sum_{i=1}^{d}
\left|\int_{0}^{T}\frac{Q u(s, i)}{u(s, i)} [\one_{\{X_{s}^{\epsilon}=i\}}-K_{\nu}(s, i)]\DD s \right|-\frac{1}{\epsilon}\sum_{i=1}^{d}
\int_{0}^{T}\frac{Q u(s, i)}{u(s, i)} K_{\nu}(s, i)\DD s\right)
\end{eqnarray*}
On the set 
$\{\rho_{T}(M^{\epsilon}, \varphi)+d_T(\nu^{\epsilon}, \nu)\leqslant \delta\}$, Lemma \ref{integrationbyparts} implies that, for any $\gamma>0$, $i\in \mathbb{S}$,
\[\left|\int_{0}^{T}\frac{\frac{\partial}{\partial s} u(s, i)}{u(s, i)} [\one_{\{X_{s}^{\epsilon}=i\}}-K_{\nu}(s, i)]\DD s\right| \leqslant C_u\delta + \gamma, \]
\[\left|\int_{0}^{T}\frac{Q u(s, i)}{u(s, i)} [\one_{\{X_{s}^{\epsilon}=i\}}-K_{\nu}(s, i)]\DD s \right|\leqslant K_{Q, u}\delta, \:\:\: \forall i \in \mathbb{S}.\] 
Since $\frac{\partial}{\partial s} u(s, i)/u(s, i)$ is continuous on $[0, T]$,
\[\left|\int_{0}^{T}\frac{\frac{\partial}{\partial s} u(s, i)}{u(s, i)} K_{\nu}(s, i)\DD s\right|\leqslant C'_{u}T. \]
Hence,
\[
\eee(\tilde{N}^{\epsilon})_T\geqslant
K_u\exp\left(-C_u\delta d-\gamma d-C'_uTd-\frac{1}{\epsilon}K_{Q, u} \delta d-\frac{1}{\epsilon}
\int_{0}^{T}\sum_{i=1}^{d}\frac{Q u(s, i)}{u(s, i)} K_{\nu}(s, i)\DD s\right).
\] We have thus proven our claim.
\end{proof}

Now we are ready to prove the upper bound in the local LDP.

\begin{proposition} \label{upp}
For every $(\varphi, \nu) \in \mathbb{C}_T \times \mathbb{M}_T$,
\[\limsup_{\delta\rightarrow 0}\limsup_{\epsilon \rightarrow 0} \epsilon \log \PP
(\rho_T(M^{\epsilon}, \varphi)+d_T(\nu^{\epsilon}, \nu)\leqslant \delta)\leqslant -L_{T}(\varphi, \nu).\]
\end{proposition}
\begin{proof}
Due to Lemma \ref{dense}, $\mathbb{C}_T \times \mathbb{M}^{++}_T$ is dense in $\mathbb{C}_T \times \mathbb{M}_T$. We first prove that the upper bound holds on $\mathbb{C}_T \times \mathbb{M}^{++}_T$. For every $\nu \in \mathbb{M}^{++}_T$, it is an immediate implication of Lemma \ref{optimizer} that there is an optimizer $u^*(\cdot,  \cdot)$ of \[\inf_{u \in U}\left[\sum_{i=1}^{d}\frac{(Qu)(i)}{u(i)}K_{\nu}(s, i)\right]\] such that $u^* \in \mathbb{U}$.  We denote 
\[\eee^{u^*}_t=\frac{u^*(t, X^{\epsilon}_{t})}{u^*(0, X^{\epsilon}_{0})}\exp\left(-\int_{0}^{t}\frac{\frac{\partial}{\partial s} u^*(s, X^{\epsilon}_s)+(Q^{\epsilon}u^*)
(s, X^{\epsilon}_{s})}{u^*(s, X_{s}^{\epsilon})}\DD s\right).\]
Lemma \ref{supermartingale} implies that 
\begin{equation}\label{ub}
\ee\left[\one_{\{\rho_T(M^{\epsilon}, \varphi)
+d_T(\nu^{\epsilon}, \nu)\leqslant \delta\}}\eee^{u^*}_T \eee(N^{\epsilon})_T\right]
\leqslant 1
\end{equation}
for every $\lambda \in \mathbb{S}_T.$
By virtue of Lemmas \ref{nonrandom bound 1} and \ref{nonrandom bound 2}, we have a non-random lower bound for 
$\eee^{u^*}_T \eee(N^{\epsilon})_T$ on the set 
$\{\rho_{T}(M^{\epsilon}, \varphi)+d_T(\nu^{\epsilon}, \nu)\leqslant \delta\}$.  Hence, (\ref{ub}) implies that, for
all $\lambda \in \mathbb{S}_T$,
\begin{eqnarray*}
 \lefteqn{\PP(\rho_T(M^{\epsilon}, \varphi)+d_T(\nu^{\epsilon}, \nu)\leqslant \delta)}\\
&\leqslant& \frac{1}{K_{u^*}} \exp\left(C_{u^*}\delta d+C'_{u^*}Td+\frac{1}{\epsilon}K_{Q, u^*}\delta d+\frac{1}{\epsilon}
\int_{0}^{T} \sum_{i=1}^{d} \frac{Q u^*(s, i)}{u^*(s, i)} K_{\nu}(s, i)\DD s\right) \times\\
& &\exp\left\{-\frac{1}{\epsilon}\left[\int_{0}^{T}\lambda(s)\DD \varphi_{s}
-\int_{0}^{T}\lambda(s)\hat{b}( \nu, \varphi_{s})\DD s
-\int_{0}^{T}\frac{\lambda^2(s)}{2}\hat{\sigma}^2( \nu, \varphi_{s})\DD s\right]+\frac{\delta}{\epsilon} K_{\lambda, \varphi, T}\right\}.
\end{eqnarray*}
We observe that 
\[ \int_{0}^{T} \sum_{i=1}^{d} \frac{Q u^*(s, i)}{u^*(s, i)} K_{\nu}(s, i)\DD s=-\int_0^T
 \sup_{u \in U}\left[-\sum_{i=1}^{d}\frac{(Qu)(i)}{u(i)}K_{\nu}(s, i)\right] \DD s=-\tilde{I}_{T}(\nu).\]
It directly entails that, again for all $\lambda \in \mathbb{S}_T$,
\begin{eqnarray} \lefteqn{
\epsilon \log \PP(\rho_T(M^{\epsilon}, \varphi)+d_T(\nu^{\epsilon}, \nu)\leqslant \delta)}\nonumber\\
&\leqslant& -\left[\int_{0}^{T}\lambda(s)\DD \varphi_{s}
-\int_{0}^{T}\lambda(s)\hat{b}( \nu, \varphi_{s})\DD s
-\int_{0}^{T}\frac{\lambda^2(s)}{2}\hat{\sigma}^2( \nu, \varphi_{s})\DD s\right]+K_{\lambda, \varphi, T}\delta\nonumber \\
& & 
-\epsilon \log K_{u^*}+\epsilon\left(C_{u^*}\delta d+C'_{u^*}Td\right)+K_{Q, u^*}\delta d-\tilde{I}_{T}(\nu).\label{re1}
\end{eqnarray}

It is easily seen that all the terms with $\delta$ or $\epsilon$ vanish as $\delta \rightarrow 0$, $\epsilon \rightarrow 0$. As a consequence we conclude, by minimizing the right hand-side over $\lambda$, that the decay rate
\[\limsup_{\delta\rightarrow 0}\limsup_{\epsilon \rightarrow 0} \epsilon \log \PP
(\rho_T(M^{\epsilon}, \varphi)+d_T(\nu^{\epsilon}, \nu)\leqslant \delta)\] is majorized by
\[
- \sup_{\lambda \in \mathbb{S}_T}\left[\int_{0}^{T}\lambda(s)\DD \varphi_{s}
-\int_{0}^{T}\lambda(s)\hat{b}( \nu, \varphi_{s})\DD s
-\int_{0}^{T}\frac{\lambda^2(s)}{2}\hat{\sigma}^2( \nu, \varphi_{s})\DD s\right]-\tilde{I}_{T}(\nu).
\]
Since $b(i, x)$ and $\sigma(i, x)$ satisfy the linear growth condition (A.2), $\hat{b}(\nu, x)$ and $\hat{\sigma}(\nu, x)$ are of linear growth as well.
Then Liptser and Pukhalskii \cite[Lemma 6.1]{Liptser} implies that
\begin{eqnarray*}
\lefteqn{\sup_{\lambda \in \mathbb{S}_T}\left[\int_{0}^{T}\lambda(s)\DD \varphi_{s}
-\int_{0}^{T}\lambda(s)\hat{b}(\nu, \varphi_{s})\DD s
-\int_{0}^{T}\frac{\lambda^2(s)}{2}\hat{\sigma}^2(\nu, \varphi_{s})\DD s\right]}\\
&=&
\left\{ {\displaystyle 
  \begin{array}{l l}{\displaystyle 
\int_{0}^{T}\sup_{\lambda \in \rr}\left[\lambda \varphi'_{s}
-\lambda\hat{b}(\nu, \varphi_{s})
-\frac{\lambda^2}{2}\hat{\sigma}^2(\nu, \varphi_{s})\right]\DD s }& \quad 
 \text{if } \varphi \in \mathbb{H}_{T},\\
    \infty & \quad \text{otherwise.}
  \end{array} }\right.
\end{eqnarray*}

For $s \in [0, T]$ such that $\hat{\sigma}^2(\nu, \varphi_{s}) \neq 0$ and $\varphi \in \mathbb{H}_T$,
it is well-known (cf. Fredlin and Wentzell \cite{Freidlin}, Liptser \cite{Liptser1}) that
\[\sup_{\lambda \in \rr}\left[\lambda \varphi'_{s}-\lambda\hat{b}(\nu, \varphi_{s})
-\frac{\lambda^2}{2}\hat{\sigma}^2(\nu, \varphi_{s})\right]= 
\frac{[\varphi'_{s}-\hat{b}(\nu, \varphi_{s})]^2}{2\hat{\sigma}^2(\nu, \varphi_{s})}.\]
For $s \in [0, T]$ such that $\hat{\sigma}^2(\nu, \varphi_{s}) = 0$ and $\varphi \in \mathbb{H}_T$,
\begin{equation*}
\sup_{\lambda \in \rr}\left[\lambda \varphi'_{s}-\lambda\hat{b}(\nu, \varphi_{s})
-\frac{\lambda^2}{2}\hat{\sigma}^2(\nu, \varphi_{s})\right] = \left\{
  \begin{array}{l l}
 0 & \quad 
 \text{if } \varphi'_s=\hat{b}(\nu, \varphi_s),\\
    \infty & \quad \text{otherwise.}
  \end{array} \right.
\end{equation*}
Hence, with the conventions $0/0 = 0$ and $n/0 = \infty$ (for all $n > 0$) being in force, 
\[\int_{0}^{T}\sup_{\lambda \in \rr}\left[\lambda \varphi'_{s}
-\lambda\hat{b}(\nu, \varphi_{s})
-\frac{\lambda^2}{2}\hat{\sigma}^2(\nu, \varphi_{s})\right]\DD s=
\frac{1}{2} \int_{0}^{T} \frac{[\varphi'_{t}-\hat{b}(\nu, \varphi_{t})]^2}{\hat{\sigma}^2(\nu, \varphi_{t})}\DD t\] 
if $\varphi \in \mathbb{H}_T$.

Hence the lower bound for the dense subset $\mathbb{C}_T \times \mathbb{M}^{++}_T$ is established.  In consideration of Lemma \ref{mogu}, the upper bound is proved for $ \mathbb{C}_T \times \mathbb{M}_T$ if we can show $I_{T}(\varphi, \nu)$ and $\tilde{I}_{T}(\nu)$ are lower semi-continuous on $\nu$.  We denote
\[F_{\lambda}(\varphi, \nu)=\int_{0}^{T}\lambda(s)\DD \varphi_{s}
-\int_{0}^{T}\lambda(s)\hat{b}( \nu, \varphi_{s})\DD s
-\int_{0}^{T}\frac{\lambda^2(s)}{2}\hat{\sigma}^2( \nu, \varphi_{s})\DD s.\]
By the above computation, we know for every $(\varphi, \nu) \in \mathbb{C}_T \times \mathbb{M}_T$, 
$I_T(\varphi, \nu)=\sup_{\lambda \in \mathbb{S}_T}F_{\lambda}(\varphi, \nu)$. For every $\lambda \in \mathbb{S}_T$, $F_{\lambda}(\nu, \varphi)$ is continuous on $\nu$ due to Lemma \ref{integrationbyparts}. Then $I_T(\varphi, \nu)$ is lower semi-continuous on $\nu$ since it is the pointwise supremum of continuous functions. By Lemma \ref{continuity}, $\tilde{I}_{T}(\nu)$ also satisfies the requirement.  The claim is established.
\end{proof}

\section{Lower bound for the local LDP}
This section studies the lower bound of the local LDP. To this end, it is realized that only finite rate functions need to be investigated.  The rate function $\tilde{I}_T(\nu)$ is finite for every $\nu \in \mathbb{M}_T$ since
$0 \leqslant \tilde{I}_T(\nu) \leqslant -T\sum_{i=1}^{d}Q_{ii}$. We further observe that the rate function $I_T(\varphi, \nu)$ is finite for every $(\varphi, \nu) \in \mathbb{H}_T \times \mathbb{M}_T$
if $\inf_{i, x}\sigma^2(i , x) > 0$. Hence we consider the case of $\inf_{i, x}\sigma^2(i , x) > 0$ first. Let $(\varphi, \nu) \in \mathbb{H}_T \times \mathbb{M}_T$. We define
\begin{equation} \label{Girs}
\bar{N}_{t}^{\epsilon}:=\frac{1}{\sqrt{\epsilon}} 
\int_{0}^{t} \frac{\varphi'_{s}-\hat{b}(\varphi_{s},\nu)}{\hat{\sigma}(\varphi_{s}, \nu)} \DD B_s .
\end{equation}
Then its stochastic exponential is
\[\eee(\bar{N}^{\epsilon})_{t}=\exp\left(\bar{N}_{t}^{\epsilon}-\frac{1}{2}\langle \bar{N}^{\epsilon}\rangle_{t}\right),\:\:\:\:\:
\mbox{where}  \:\:\:\:\:
\langle \bar{N}^{\epsilon}\rangle_{t}=\frac{1}{\epsilon} \int_{0}^{t} \left[\frac{\varphi'_{s}-\hat{b}(\varphi_{s},\nu)}{\hat{\sigma}(\varphi_{s}, \nu)}\right]^2 \DD s .\]
For simplicity, we denote 
\[h_s:=\frac{\varphi'_{s}-\hat{b}(\varphi_{s},\nu)}{\hat{\sigma}(\varphi_{s}, \nu)}\]
throughout this section. Recall from (\ref{N}) that, for a given $u(\cdot, \cdot) \in \mathbb{U}$,
\[\eee(\tilde{N}^{\epsilon})_t=\frac{u(t, X^{\epsilon}_{t})}{u(0, X^{\epsilon}_{0})}\exp\left(-\int_{0}^{t}\frac{\frac{\partial}{\partial s} u(s, X^{\epsilon}_s)+(Q^{\epsilon}u)
(s, X^{\epsilon}_{s})}{u(s, X_{s}^{\epsilon})}\DD s\right).\]

In order to perform a change of measure in Proposition \ref{non}, we show that $\{\eee(\tilde{N}^{\epsilon})_t \eee(\bar{N}^{\epsilon})_t\}_{t \in [0, T]}$ is a true martingale.
It is noted that in the first results of this section, we impose  the condition $\inf_{i, x}\sigma^2(i , x) > 0$, which will be lifted later on.
\begin{lemma}\label{martingale}
For every $(\varphi, \nu) \in \mathbb{H}_T \times \mathbb{M}_T$ and $u(\cdot, \cdot) \in \mathbb{U}$, $\{\eee(\tilde{N}^{\epsilon})_t \eee(\bar{N}^{\epsilon})_t\}_{t \in [0, T]}$ is a martingale if \,$\inf_{i, x}\sigma^2(i , x) > 0$.
\end{lemma}
\begin{proof}
We have shown in last section that $\eee(\tilde{N}^{\epsilon})_t$ is a martingale.
Since $\varphi\in \mathbb{H}_T$ and recalling that we assumed $\inf_{i, x}\sigma^2(i , x) > 0$,
it follows that $ \langle\bar{N}^{\epsilon}\rangle_{T}=\frac{1}{\epsilon} \int_{0}^{T} h_s^2 \DD s < \infty$.
Then Novikov's condition implies that $\eee(\bar{N}^{\epsilon})_t$ is also a martingale.
Since $X_{t}^{\epsilon}$ is independent of the Brownian motion $B_t$, $\eee(\tilde{N}^{\epsilon})_t$ is also
independent of $\eee(\bar{N}^{\epsilon})_t$. So, 
\[\ee[\eee(\tilde{N}^{\epsilon})_T \eee(\bar{N}^{\epsilon})_T]=\ee[\eee(\tilde{N}^{\epsilon})_T] \ee[ \eee(\bar{N}^{\epsilon})_T]
=\ee[\eee(\tilde{N}^{\epsilon})_0] \ee[ \eee(\bar{N}^{\epsilon})_0]=\ee[\eee(\tilde{N}^{\epsilon})_0 \eee(\bar{N}^{\epsilon})_0]. \]
By the same reasoning as in the proof of Lemma \ref{supermartingale}, we know that $\eee(\tilde{N}^{\epsilon})_t \eee(\bar{N}^{\epsilon})_t$ is a 
supermartingale. Hence, it is a martingale by Liptser and Shiryaev \cite[Lemma 6.4]{Shiryaev}.
\end{proof}

\begin{lemma}\label{nonrandom bound 3}
For every $ \nu \in \mathbb{M}_T$, every $u \in \mathbb{U}$ and every $\gamma, \delta>0$,  there exist positive constants
$C_u$, $C'_u$, $K'_u$ and $K_{Q, u}$ not depending on $\epsilon$ or $\delta$ such that
 \[[\eee(\tilde{N}^{\epsilon})_T]^{-1}\geqslant K'_u\exp\left(-C_u\delta d-\gamma d-C'_uTd-\frac{1}{\epsilon}K_{Q, u} \delta d+\frac{1}{\epsilon}
\int_{0}^{T} \sum_{i=1}^{d} \frac{Q u(s, i)}{u(s, i)} K_{\nu}(s, i)\DD s\right)\]
on the set 
$\{\rho_{T}(M^{\epsilon}, \varphi)+d_T(\nu^{\epsilon}, \nu)\leqslant \delta\}.$
\end{lemma}
\begin{proof}
According to the computation in Lemma \ref{nonrandom bound 2}, we have $[\eee(\tilde{N}^{\epsilon})_T]^{-1}$ is equal to
\begin{eqnarray*}
&&\frac{u(0, X^{\epsilon}_{0})}{u(T, X^{\epsilon}_{T})}\exp\left(\sum_{i=1}^{d}
\int_{0}^{T}\frac{\frac{\partial}{\partial s} u(s, i)}{u(s, i)} [\one_{\{X_{s}^{\epsilon}=i\}}-K_{\nu}(s, i)]\DD s+\sum_{i=1}^{d}
\int_{0}^{T}\frac{\frac{\partial}{\partial s} u(s, i)}{u(s, i)} K_{\nu}(s, i)\DD s\right) \times\\
&&  \exp\left(\frac{1}{\epsilon}\sum_{i=1}^{d}
\int_{0}^{T}\frac{Q u(s, i)}{u(s, i)} [\one_{\{X_{s}^{\epsilon}=i\}}-K_{\nu}(s, i)]\DD s+\frac{1}{\epsilon}\sum_{i=1}^{d}
\int_{0}^{T}\frac{Q u(s, i)}{u(s, i)} K_{\nu}(s, i)\DD s\right)\end{eqnarray*}
Defining $K'_u:=\min_{i, j} u(0, j)/u(T, i)$, we have that 
$[\eee(\tilde{N}^{\epsilon})_T]^{-1}$ is not less than 
\begin{eqnarray*}
&&K'_u\exp\left(-\sum_{i=1}^{d}
\left|\int_{0}^{T}\frac{\frac{\partial}{\partial s} u(s, i)}{u(s, i)} [\one_{\{X_{s}^{\epsilon}=i\}}-K_{\nu}(s, i)]\DD s\right|-\sum_{i=1}^{d}
\left|\int_{0}^{T}\frac{\frac{\partial}{\partial s} u(s, i)}{u(s, i)} K_{\nu}(s, i)\DD s\right|\right) \times\\
&& \,\ \exp\left(-\frac{1}{\epsilon}\sum_{i=1}^{d}
\left|\int_{0}^{T}\frac{Q u(s, i)}{u(s, i)} [\one_{\{X_{s}^{\epsilon}=i\}}-K_{\nu}(s, i)]\DD s \right|+\frac{1}{\epsilon}\sum_{i=1}^{d}
\int_{0}^{T}\frac{Q u(s, i)}{u(s, i)} K_{\nu}(s, i)\DD s\right)
\end{eqnarray*}
On the set 
$\{\rho_{T}(M^{\epsilon}, \varphi)+d_T(\nu^{\epsilon}, \nu)\leqslant \delta\}$, Lemma \ref{integrationbyparts} implies 
\[\left|\int_{0}^{T}\frac{\frac{\partial}{\partial s} u(s, i)}{u(s, i)} [\one_{\{X_{s}^{\epsilon}=i\}}-K_{\nu}(s, i)]\DD s\right| \leqslant C_u\delta + \gamma, \:\:\: \forall i \in \mathbb{S},  \forall \gamma>0,\]
\[\left|\int_{0}^{T}\frac{Q u(s, i)}{u(s, i)} [\one_{\{X_{s}^{\epsilon}=i\}}-K_{\nu}(s, i)]\DD s \right|\leqslant K_{Q, u}\delta, \:\:\: \forall i \in \mathbb{S}.\]
Also, 
\[\left|\int_{0}^{T}\frac{\frac{\partial}{\partial s} u(s, i)}{u(s, i)} K_{\nu}(s, i)\DD s\right|\leqslant C'_{u}T. \]
Hence,
\[[\eee(\tilde{N}^{\epsilon})_T]^{-1}\geqslant K'_u\exp\left(-C_u\delta d-\gamma d-C'_uTd-\frac{1}{\epsilon}K_{Q, u} \delta d+\frac{1}{\epsilon}
\int_{0}^{T} \sum_{i=1}^{d} \frac{Q u(s, i)}{u(s, i)} K_{\nu}(s, i)\DD s\right).\]
We have thus derived the desired lower bound.
\end{proof}

We proceed to prove the lower bound of the local LDP under the condition $\inf_{i, x}\sigma^2(i , x) > 0$.

\begin{proposition} \label{non}
For every $(\varphi, \nu) \in \mathbb{H}_T \times \mathbb{M}_T$,  if  $\inf_{i, x}\sigma^2(i , x) > 0$,
\[\liminf_{\delta\rightarrow 0}\liminf_{\epsilon \rightarrow 0} \epsilon \log \PP
(\rho_T(M^{\epsilon}, \varphi)+d_T(\nu^{\epsilon}, \nu)\leqslant \delta)\geqslant -L_{T}(\varphi, \nu).\]
\end{proposition}
\begin{proof}
For any $\nu \in \mathbb{M}_T$, there is a sequence $\nu^{\eta} \in \mathbb{M}_T^{++}$ such that $\nu^{\eta} \rightarrow \nu$ by Lemma \ref{dense}. Actually, the convergence happens in the way that $K_{\nu}^{\eta} (\cdot, i) \rightarrow K_{\nu}(\cdot, i)$ a.e.. Then by Lemma \ref{continuity}, the rate function $L_{T}(\varphi, \nu)$ satisfies the continuity property required in Lemma \ref{mogu}. Hence we only need to prove the lower bound on the dense subset $\mathbb{H}_T \times \mathbb{M}^{++}_T$.
Recall that for every $\nu \in \mathbb{M}^{++}_T$, Lemma \ref{optimizer} implies that there is an optimizer $u^*(\cdot,  \cdot)$ of \[\inf_{u \in U}\left[\sum_{i=1}^{d}\frac{(Qu)(i)}{u(i)}K_{\nu}(s, i)\right]\] such that $u^* \in \mathbb{U}$ and 
\[\eee^{u^*}_t=\frac{u^*(t, X^{\epsilon}_{t})}{u^*(0, X^{\epsilon}_{0})}\exp\left(-\int_{0}^{t}\frac{\frac{\partial}{\partial s} u^*(s, X^{\epsilon}_s)+(Q^{\epsilon}u^*)
(s, X^{\epsilon}_{s})}{u^*(s, X_{s}^{\epsilon})}\DD s\right).\]
By Lemma \ref{martingale}, we know
$\ee[\eee^{u^*}_T \eee(\bar{N}^{\epsilon})_T]=1.$
On $(\Omega, \mathscr{F}_T)$, we define a new probability measure $\pp_{u^*}$ through
$\dd \pp_{u^*}=\eee^{u^*}_T \eee(\bar{N}^{\epsilon})_T \dd \pp.$
Since $\eee^{u^*}_T \eee(\bar{N}^{\epsilon})_T$ is strictly positive, $\pp_{u^*}$ is equivalent to $\pp$ and 
$\dd \pp=\left[\eee^{u^*}_T \eee(\bar{N}^{\epsilon})_T\right]^{-1} \dd \pp_{u^*}.$
So that we can translate the probability of our interest under the original measure ${\mathbb P}$ into the mean of a certain random quantity under the alternative measure $\pp_{u^*}$:
\begin{equation}\label{change measure}
 \PP(\rho_T(M^{\epsilon}, \varphi)+d_T(\nu^{\epsilon}, \nu)\leqslant \delta)=
\int_{\{\rho_T(M^{\epsilon}, \varphi)+d_T(\nu^{\epsilon}, \nu)\leqslant \delta\}}
\left[\eee^{u^*}_T \eee(\bar{N}^{\epsilon})_T\right]^{-1}\dd \pp_{u^*}.
\end{equation}

By Girsanov's theorem,  
$\tilde{B}_t:=B_t-\frac{1}{\sqrt{\epsilon}} \int_{0}^{t} h_s \DD s$
is a $\pp_{u^*}$-Brownian motion on $(\Omega, (\mathscr{F}_t)_{t\leqslant T})$.
We substitute the above equation in (\ref{Girs}), and obtain
\[\bar{N}_{T}^{\epsilon}-\frac{1}{2}\langle\bar{N}^{\epsilon}\rangle_{T}=\frac{1}{\sqrt{\epsilon}} 
\int_{0}^{T} h_s\DD \tilde{B}_s+ \frac{1}{2\epsilon} \int_{0}^{T} h_s^2 \DD s.\]
It thus follows that $[\eee^{u^*}_T \eee(\bar{N}^{\epsilon})_T]^{-1}$ is equal to 
\[\frac{u^*(0, X^{\epsilon}_{0})}{u^*(t, X^{\epsilon}_{t})}\exp\left(\int_{0}^{t}\frac{\frac{\partial}{\partial s} u^*(s, X^{\epsilon}_s)+(Q^{\epsilon}u^*)
(s, X^{\epsilon}_{s})}{u^*(s, X_{s}^{\epsilon})}\DD s-\frac{1}{\sqrt{\epsilon}} 
\int_{0}^{T} h_s \DD \tilde{B}_s 
-\frac{1}{2\epsilon} \int_{0}^{T} h_s^2 \DD s \right).\]
Now let $L$ be a positive constant. We define
$\Theta^{\epsilon}:=\left\{\rho_T(M^{\epsilon}, \varphi)+d_T(\nu^{\epsilon}, \nu)\leqslant \delta, \left|\int_{0}^{T} h_s \DD \tilde{B}_s\right|\leqslant L\right\}.$
Then (\ref{change measure}) implies 
\[\PP(\rho_T(M^{\epsilon}, \varphi)+d_T(\nu^{\epsilon}, \nu)\leqslant \delta)\geqslant
\int_{\Theta^{\epsilon}}
\left[\eee^{u^*}_T \eee(\bar{N}^{\epsilon})_T\right]^{-1}\dd \pp_{u^*}.\]

By Lemma \ref{nonrandom bound 3}, we obtain the following non-random lower bound  of  $[\eee^{u^*}_T \eee(\bar{N}^{\epsilon})_T]^{-1}$, valid on the set $\Theta^{\epsilon}$:
\[K'_{u^*}\exp\left(-C_{u^*}\delta d-C'_{u^*}Td-\frac{1}{\epsilon}K_{Q, {u^*}} \delta d-\frac{\tilde{I}_T(\nu)}{\epsilon}
-\frac{I_{T}(\varphi, \nu)}{\epsilon}- \frac{L}{\sqrt{\epsilon}}\right).\]
As a consequence, we have the following lower bound of the probability $\lefteqn{\PP(\rho_T(M^{\epsilon}, \varphi)+d_T(\nu^{\epsilon}, \nu)\leqslant \delta)}$:
\[K'_{u^*}\exp\left(-C_{u^*}\delta d-C'_{u^*}Td-\frac{1}{\epsilon}K_{Q, {u^*}} \delta d-\frac{\tilde{I}_T(\nu)}{\epsilon}-\frac{I_{T}(\varphi, \nu)}{\epsilon}- \frac{L}{\sqrt{\epsilon}}\right)\times \pp_{u^*}(\Theta^{\epsilon}).\]
This, in turn, leads to the following lower bound on the corresponding exponential decay rate:
\begin{eqnarray} \label{ustar}
\epsilon \log \PP(\rho_T(M^{\epsilon}, \varphi)+d_T(\nu^{\epsilon}, \nu)\leqslant \delta) 
&\geqslant& \epsilon \log K'_{u^*}-\epsilon (C_{u^*}\delta d+C'_{u^*}Td)-K_{Q, {u^*}} \delta d \nonumber \\
& & -\tilde{I}_T(\nu)-I_{T}(\varphi, \nu)- \sqrt{\epsilon}L +\, \epsilon \log \pp_{u^*}(\Theta^{\epsilon}).
\end{eqnarray} 

Then a sufficient condition for  desired result to hold is 
$\lim_{\epsilon \rightarrow 0} \pp_{u^*}(\Theta^{\epsilon})>0.$
It is evident that
\begin{equation}\label{threep}\pp_{u^*}(\Theta^{\epsilon}) \geqslant1
-\pp_{u^*}\left(\left|\int_{0}^{T} h_s \DD \tilde{B}_s\right|> L\right)
-\pp_{u^*}(d_T(\nu^{\epsilon}, \nu)> \delta)
-\pp_{u^*}(\rho_T(M^{\epsilon}, \varphi)>\delta).
\end{equation}
We proceed by consecutively proving that the three probabilities in the right-hand side of (\ref{threep}) vanishes as $\epsilon \rightarrow 0$.
We start by analyzing the first probability. By Chebyshev's inequality,
\[\tilde{\pp}_{u^*}\left(\left|\int_{0}^{T} h_s \DD \tilde{B}_s\right|> L\right)\leqslant 
\frac{\tilde{\ee}_{u^*}\left|\int_{0}^{T} h_s \DD \tilde{B}_s\right|^2}{L^2}
=\frac{\int_{0}^{T} h_s^2 \DD s}{L^2}.\]
Since $\int_{0}^{T} h_s^2 \DD s<\infty$, we can make this upper bound arbitrarily small by picking $L$ sufficiently large.

Next we consider the second probability in the right-hand sider of (\ref{threep}).  We notice that the part \[\exp \left(\frac{1}{\sqrt{\epsilon}} \int_{0}^{T} h_s \DD \tilde{B}_s 
+\frac{1}{2\epsilon} \int_{0}^{T} h_s^2 \DD s \right)\] in the change of measure procedure is not related to the Markov chain . Then by 
 Proposition 11.2.3 in Bielecki and Rutkowski \cite{bielecki}, a Markov chain $X_t$ with transition intensity matrix $Q$ under $\pp$ becomes a Markov
chain under $\pp_{u^*}$ with transition intensity matrix
$Q(u^*)(t)$ where \[Q(u^*)(t)_{ij}=Q_{ij}\frac{u^*(t, j)}{u^*(t, i)} \:\:\: \text{for $i\not=j$; } \:\:\:Q(u^*)(t)_{ii}=-\sum_{j\neq i}Q_{ij}\frac{u^*(t, j)}{u^*(t, i)}.\]
Hence, $Q(u^*)(t)/\epsilon$ is the transition
intensity matrix of $X^{\epsilon}_{t}$ under $\pp_{u^*}$.
By Lemma \ref{invariant}, for every 
$t \in [0, T]$, $\boldsymbol{K}_{\nu}(t)= (K_{\nu}(t, 1), \ldots, K_{\nu}(t, d))$ is the unique solution 
of \[\mu(t) Q(u^*)(t)=0,\:\:\:\: \sum_{i=1}^{d} \mu(t, i)=1,\:\:\:\: \mu(t, i) \geqslant 0.\] 
Also, all entries of the matrix $Q(u^*)(t)$ are smooth on $[0, T]$ by Lemma \ref{optimizer}.
Then by Corollary 5.8 in Yin and Zhang \cite{Yin_1}, 
\[\pp_{u^*} \left(\sup_{t \leqslant T, i \in \mathbb{S}}\left|\int_{0}^{t} \one_{\{X_{s}^{\epsilon}=i\}} \DD s-\int_{0}^{t} K_{\nu}(i, s) \DD s\right|
> \epsilon^{1/4}\right) \leqslant K\exp\left\{-\frac{C_{T}}{\epsilon^{1/4}(T+1)^{3/2}}\right\},\]
where $C_T$ is a strictly positive constant. 
For any $\delta>0$, the following (obvious) inequality holds for every $\epsilon$ such that $\epsilon \in (0, \delta^{4})$
\begin{eqnarray*}
\lefteqn{\pp_{u^*} \left(\sup_{t \leqslant T, i \in \mathbb{S}}\left|\int_{0}^{t} \one_{\{X_{s}^{\epsilon}=i\}} \DD s-\int_{0}^{t} K_{\nu}(i, s) \DD s\right|> \delta \right) }\\
&\leqslant&
\pp_{u^*} \left(\sup_{t \leqslant T, i \in \mathbb{S}}\left|\int_{0}^{t} \one_{\{X_{s}^{\epsilon}=i\}} \DD s-\int_{0}^{t} K_{\nu}(i, s) \DD s\right|> \epsilon^{1/4}\right).\end{eqnarray*}
Hence we have 
\[\pp_{u^*} \left(\sup_{t \leqslant T, i\in \mathbb{S}}\left|\int_{0}^{t} \one_{\{X_{s}^{\epsilon}=i\}} \DD s-\int_{0}^{t} K_{\nu}(i, s) \DD s\right|
> \delta\right) \rightarrow 0 \text{ as } \epsilon \rightarrow 0.\]
That is, $\pp_{u^*}(d_T(\nu^{\epsilon}, \nu)> \delta)\rightarrow 0, \text{ as } \epsilon \rightarrow 0$.

Now we proceed by showing the third probability in the right-hand side of (\ref{threep}) vanishes as $\epsilon\to 0$. We substitute $\tilde{B}_t$ for $B_t$ in (\ref{def}), yielding
\[M^{\epsilon}_{t}=\int_{0}^{t}b(X_{s}^{\epsilon}, M_{s}^{\epsilon})+h_s\sigma(X_{s}^{\epsilon}, M_{s}^{\epsilon})\DD s+
\sqrt{\epsilon}\int_{0}^{t}\sigma(X_{s}^{\epsilon}, M_{s}^{\epsilon})\DD \tilde{B}_{s}.\]
By setting $\tilde{M}_t^{\epsilon}:=M^{\epsilon}_t-\varphi_t,$ we obtain
\begin{eqnarray*}
\tilde{M}_t^{\epsilon} &=&\sqrt{\epsilon}\int_{0}^{t}\sigma(X_{s}^{\epsilon}, M_{s}^{\epsilon})\DD \tilde{B}_{s}+\int_{0}^{t}[b(X_{s}^{\epsilon}, M_{s}^{\epsilon})-b(X_{s}^{\epsilon}, \varphi_s)]\DD s
+\int_{0}^{t}[b(X_{s}^{\epsilon}, \varphi_s)-\hat{b}(\nu, \varphi_s)] \DD s\\
&&+\:\int_{0}^{t}h_s[\sigma(X_{s}^{\epsilon}, M_{s}^{\epsilon})-\sigma(X_{s}^{\epsilon}, \varphi_s)]\DD s+
\int_{0}^{t}h_s[\sigma(X_{s}^{\epsilon}, \varphi_s)-\hat{\sigma}(\nu, \varphi_s)]\DD s.
\end{eqnarray*}
Using the Lipschitz continuity featuring in (A1), we find that both
\[\sup_{t\leqslant T}\left|\int_{0}^{t}[b(X_{s}^{\epsilon}, M_{s}^{\epsilon})-b(X_{s}^{\epsilon}, \varphi_s)]\DD s \right|
\leqslant \int_0^T K \tilde{M}_s^{\epsilon*} \DD s, \]
and\[\sup_{t\leqslant T}\left|\int_{0}^{t} h_s[\sigma(X_{s}^{\epsilon}, M_{s}^{\epsilon})-\sigma(X_{s}^{\epsilon}, \varphi_s)]  \DD s \right|
\leqslant \int_0^T |h_s|K \tilde{M}_s^{\epsilon*} \DD s.\]
Recalling that we denote throughout this paper running maximum processes by adding an asterisk (`$^*$'), it is now immediate that
\[
\tilde{M}_T^{\epsilon*} \leqslant I_T^{1*}+I_T^{2*}+I_T^{3*}+ \int_0^T K(1+|h_s|) \tilde{M}_s^{\epsilon*}\DD s, 
\]
where
\begin{eqnarray*}I_t^1&:=&\sqrt{\epsilon}\int_{0}^{t}\hspace{-1.5mm}\sigma(X_{s}^{\epsilon}, M_{s}^{\epsilon})\DD \tilde{B}_{s},\:\:\:\:\:
I_t^2:=\int_{0}^{t}\hspace{-1.5mm}[b(X_{s}^{\epsilon}, \varphi_s)-\hat{b}(\nu, \varphi_s) ]\DD s,\\
I_t^3&:=&\int_{0}^{t}\hspace{-1.5mm}h_s[\sigma(X_{s}^{\epsilon}, \varphi_s)-\hat{\sigma}(\nu, \varphi_s)]\DD s.\end{eqnarray*}
Then Gronwall's inequality implies
\begin{equation}\label{gr}
\tilde{M}_T^{\epsilon*} \leqslant [I_T^{1*}+I_T^{2*}+I_T^{3*}] \exp \left(\int_0^T K(1+|h_s|) \DD s \right). 
\end{equation}

The next step is to study the impact of $I_T^{1*}$, $I_T^{2*}$, and $I_T^{3*}$ separately.
For any $\delta>0$, it is an immediate consequence of Chebyshev's inequality that 
$\pp_{u^*}(I_T^{1*}>\delta)\leqslant {\delta^{-3}}\tilde{\ee}_{u^*}[(I_T^{1*})^3].$
We notice the close similarity between $I_t^1$ and $C_t^{\epsilon}$ in the proof of Proposition \ref{expotight}. The quantity
$\ee_{u^*}[(I_T^{1*})^3]$ can be dealt with using essentially the same procedure that was used to bound $\ee[(C_T^{\epsilon*})^{1/\epsilon}]$:
we derive an inequality similar to (\ref{local maximal}), i.e.,
\[\ee_{u^*}[(I_T^{1*})^3] \leqslant 
\frac{27}{8}\ee_{u^*}\left[3\epsilon \int_0^T |I_s^1|\sigma^2(X_{s}^{\epsilon}, M_{s}^{\epsilon})\DD s \right].\]
We thus obtain
\[\ee_{u^*}[(I_T^{1*})^3] \leqslant \frac{81 \epsilon kT}{8}\exp\left(\frac{81 \epsilon kT}{8}\right),\]
where $k$ is a positive constant. As a consequence, $\lim_{\epsilon \rightarrow 0} \pp_{u^*}(I_T^{1*}>\delta)=0.$

The claim $\lim_{\epsilon \rightarrow 0} \pp_{u^*}(I_T^{2*}>\delta)=0$ can be established as follows.  As a first stap we observe that since
\begin{eqnarray*}
 \sup_{t \leqslant T}\left|\int_{0}^{t}[b(X_{s}^{\epsilon}, \varphi_s)-\hat{b}(\nu, \varphi_s) ]\DD s \right|
&=&  \sup_{t \leqslant T}\left|\sum_{i=1}^{d}\int_0^t b(i, \varphi_s)[\one_{\{X^{\epsilon}_s=i\}}-K_{\nu}(s, i)]\DD s\right|\\
&\leqslant& d  \sup_{t \leqslant T, i \in \mathbb{S}}\left|\int_0^t b(i, \varphi_s)[\one_{\{X^{\epsilon}_s=i\}}-K_{\nu}(s, i)]\DD s\right|,
\end{eqnarray*}
the following upper bound applies:
\begin{eqnarray*}
\pp_{u^*}(I_T^{2*}>\delta)
&=&\pp_{u^*} \left( \sup_{t \leqslant T}\left|\int_{0}^{t}[b(X_{s}^{\epsilon}, \varphi_s)-\hat{b}(\nu, \varphi_s) ]\DD s \right|>\delta\right)\\
&\leqslant& \pp_{u^*} \left( \sup_{t \leqslant T, i \in \mathbb{S}}\left|\int_0^t b(i, \varphi_s)[\one_{\{X^{\epsilon}_s=i\}}-K_{\nu}(s, i)]\DD s\right|> \delta/d\right)
\end{eqnarray*}
Since $b(i, x)$ is Lipschitz continuous in $x$ and $\varphi_t$ is absolutely continuous, $b(i, \varphi_t)$ is bounded on $[0, T]$. Then by Corollary 5.8 in Yin and Zhang \cite{Yin_1} again, 
$\pp_{u^*}(I_T^{2*}>\delta) \rightarrow 0, \text{as } \epsilon \rightarrow 0$.

Similar to the above computation, we can obtain that 
\[
\pp_{u^*}(I_T^{3*}>\delta)
\leqslant \pp_{u^*} \left( \sup_{t \leqslant T, i \in \mathbb{S}}\left|\int_0^t h_s \sigma(i, \varphi_s)[\one_{\{X^{\epsilon}_s=i\}}-K_{\nu}(s, i)]\DD s\right|> \delta/d\right).
\]
We know that $h_s \sigma(i, \varphi_s)$ is square-integrable for every $i \in \mathbb{S}$. Then by the method of mollification in Theorem C.6 in Evans \cite{evans}, there exists a sequence of smooth functions $h^{\eta}(i, s)$ such that $h^{\eta}(i, s) \rightarrow h_s \sigma(i, \varphi_s)$ as $\eta \rightarrow 0$ in $L^2[0, T]$. By the Cauchy-Schwarz inequality,
\[
\left|\int_0^t [h_s \sigma(i, \varphi_s)-h^{\eta}(i,s)][\one_{\{X^{\epsilon}_s=i\}}-K_{\nu}(s, i)]\DD s\right|
\leqslant \sqrt{2t} \left(\int_0^t [h_s \sigma(i, \varphi_s)-h^{\eta}(i,s)]^2\DD s\right)^{1/2}.
\]
Then,
\begin{eqnarray*}
\lefteqn{\hspace{3mm}\sup_{t \leqslant T, i \in \mathbb{S}}\left|\int_0^t h_s \sigma(i, \varphi_s)[\one_{\{X^{\epsilon}_s=i\}}-K_{\nu}(s, i)]\DD s\right|}\\
&\leqslant& \sup_{t \leqslant T, i \in \mathbb{S}}\left|\int_0^t h^{\eta}(i,s)[\one_{\{X^{\epsilon}_s=i\}}-K_{\nu}(s, i)]\DD s\right|
+ \sup_{ i \in \mathbb{S}} \sqrt{2T} \left(\int_0^T [h_s \sigma(i, \varphi_s)-h^{\eta}(i,s)]^2\DD s\right)^{1/2}.
\end{eqnarray*}
We let $H(\eta):= \sup_{ i \in \mathbb{S}}  \sqrt{2T}\left(\int_0^T [h_s \sigma(i, \varphi_s)-h^{\eta}(i,s)]^2\DD s\right)^{1/2}$. It is clear that $H(\eta) \rightarrow 0$ as $\eta \rightarrow 0$.
Hence,
\[
\pp_{u^*}(I_T^{3*}>\delta)
\leqslant \pp_{u^*} \left( \sup_{t \leqslant T, i \in \mathbb{S}}\left|\int_0^t h^{\eta}(i, s)[\one_{\{X^{\epsilon}_s=i\}}-K_{\nu}(s, i)]\DD s\right|+H(\eta)> \delta/d\right).
\]
For any $\delta>0$, we can choose all $\eta >0$ small enough such that $H(\eta)< \delta/2d$. It yields
\[
\pp_{u^*}(I_T^{3*}>\delta)
\leqslant \pp_{u^*} \left( \sup_{t \leqslant T, i \in \mathbb{S}}\left|\int_0^t h^{\eta}(i, s)[\one_{\{X^{\epsilon}_s=i\}}-K_{\nu}(s, i)]\DD s\right|> \delta/2d\right).\]
Since $h^{\eta}(i, s)$ is bounded on $[0, T]$, the probability in the right-hand of above inequality vanishes as $\epsilon \rightarrow 0$ for any small enough $\eta$ by Corollary 5.8 in Yin and Zhang \cite{Yin_1}. Hence we conclude that $\lim_{\epsilon \rightarrow 0} \pp_{u^*}(I_T^{3*}>\delta)=0$.

We have thus shown that $\pp_{u^*}(\Theta^{\epsilon})$ remains bounded away from 0 as $\epsilon\to 0$. 
Upon combining all the above, the proof of the lemma is now complete.\end{proof} 

So far we have focused on the case $\inf_{i, x}\sigma^2(i , x) > 0$; to complete the analysis, we next consider the situation that this condition is lifted, in which case $\hat{\sigma}(\varphi_s, \nu)$ can be singular.
Our proof uses arguments used in the method presented by Liptser \cite[Lemma A.6]{Liptser1}.
Given $\gamma>0$, we study the stochastic differential equation
\begin{equation}\label{nonsingular}
 M_{t}^{\epsilon, \gamma}=\int_{0}^{t}b(X_{s}^{\epsilon}, M_{s}^{\epsilon, \gamma})\dd s
 +\sqrt{\epsilon}\int_{0}^{t}\sigma(X_{s}^{\epsilon}, M_{s}^{\epsilon, \gamma})\dd B_{s}
 +\sqrt{\epsilon}\gamma W_t,
\end{equation}
where $M_{0}^{\epsilon, \gamma}\equiv 0$ and $W_t$ is another standard $\pp$-Brownian motion, independent
of $B_t$ and $X_{t}^{\epsilon}$. 
We provide an auxiliary lemma which is to be used when proving the lower bound; informally, it states that $M^{\epsilon, \gamma}$ and $M^{\epsilon}$ are `superexponentially close'.

\begin{lemma} \label{gamma}
For every $T>0$ and $\eta>0$,
\begin{equation} \label{low}
 \lim_{\gamma \rightarrow 0} \limsup_{\epsilon\rightarrow 0}\epsilon \log \pp
\left(\rho_{T}(M^{\epsilon, \gamma}, M^{\epsilon})>\eta \right)=-\infty.
\end{equation}
\end{lemma}
\begin{proof}
We define
$A_{t}^{\epsilon, \gamma}:=M_{t}^{\epsilon, \gamma}-M_{t}^{\epsilon},$ and
\[\alpha^{\epsilon}_{t}:=\frac{b(X_{t}^{\epsilon}, M_{t}^{\epsilon, \gamma})-b(X_{t}^{\epsilon}, M_{t}^{\epsilon})}{M_{t}^{\epsilon, \gamma}-M_{t}^{\epsilon}},\:\:\:\:
\beta^{\epsilon}_{t}:=\frac{\sigma(X_{t}^{\epsilon}, M_{t}^{\epsilon, \gamma})-\sigma(X_{t}^{\epsilon}, M_{t}^{\epsilon})}{M_{t}^{\epsilon, \gamma}-M_{t}^{\epsilon}}.\]
By (A.1), i.e., the Lipschitz condition, we conclude
\begin{equation} \label{bou}
 |\alpha^\epsilon_t|\leqslant K,\:\:\:\: |\beta^\epsilon_t|\leqslant K, \:\:\:\:\: \forall t\in [0, T].
\end{equation}
As an immediate consequence of (\ref{def}) and (\ref{nonsingular}), we have 
 \[\dd A_{t}^{\epsilon, \gamma}=\alpha_t^\epsilon A_{t}^{\epsilon, \gamma}\dd t
 +\sqrt{\epsilon}\beta_t ^\epsilon A_{t}^{\epsilon, \gamma}\dd B_{t}
 +\sqrt{\epsilon}\gamma \dd W_t.\]
We define
\[E_{t}^{\epsilon}:=\exp\left(\int_{0}^{t}\left[\alpha_s^{\epsilon}-\frac{\epsilon}{2}(\beta_s^{\epsilon})^2\right]\dd s+\sqrt{\epsilon}
\int_0^t\beta_s^{\epsilon}\dd B_s\right).\]
We apply It\^{o}'s formula to $(E_t^{\epsilon})^{-1}$, so as to obtain
\[\dd (E_t^{\epsilon})^{-1}=(E_t^{\epsilon})^{-1}\left(\epsilon (\beta_t^{\epsilon})^2 \dd t -\alpha_t^{\epsilon}\dd t -\sqrt{\epsilon}\beta_t^{\epsilon} \dd B_t\right).\]
Since $W_t$ is independent of $B_t$, we have
\[\dd \langle A_t^{\epsilon, \gamma}, (E_t^{\epsilon})^{-1}\rangle
=\langle -\sqrt{\epsilon}\beta_t^{\epsilon} (E_t^{\epsilon})^{-1} \dd B_t, \sqrt{\epsilon}\beta_t^{\epsilon} A_{t}^{\epsilon, \gamma}\dd B_{t}
 +\sqrt{\epsilon}\gamma \dd W_t\rangle
 =-\epsilon (\beta_t^{\epsilon})^2 (E_t^{\epsilon})^{-1} A_t^{\epsilon, \gamma} \dd t.\]
By applying the integration-by-parts formula, 
\[\dd A_t^{\epsilon, \gamma}(E_t^{\epsilon})^{-1}
=A_t^{\epsilon, \gamma}\dd (E_t^{\epsilon})^{-1}+(E_t^{\epsilon})^{-1}\dd A_t^{\epsilon, \gamma}+\dd \langle A_t^{\epsilon, \gamma}, (E_t^{\epsilon})^{-1}\rangle
=\sqrt{\epsilon}\gamma (E_t^{\epsilon})^{-1}\dd W_t,\]
and hence
\[A_t^{\epsilon, \gamma}=\sqrt{\epsilon}\,\gamma E_t^{\epsilon}\int_0^t (E_s^{\epsilon})^{-1}\dd W_s.\]

We define the set $\Gamma_N:=\{1/N\leqslant \inf_{t\leqslant T}E_t^{\epsilon} \leqslant 
\sup_{t\leqslant T}E_t^{\epsilon}\leqslant N\},$ for $ N\in \mathbb{N}$.
Observe that it holds that $\rho_{T}(M^{\epsilon, \gamma}, M^{\epsilon})=(A^{\epsilon, \gamma})_T^*$,
and therefore
\begin{eqnarray*}
\pp(\rho_{T}(M^{\epsilon, \gamma}, M^{\epsilon})>\eta)
&\leqslant &\pp((A^{\epsilon, \gamma})_T^*>\eta, \Gamma_N)+\pp(\Omega\setminus \Gamma_N)\\
&\leqslant& 2 \max \left\{\pp((A^{\epsilon, \gamma})_T^*>\eta, \Gamma_N),\pp(\Omega\setminus \Gamma_N)\right\}.
\end{eqnarray*}
We now consider each of the probabilities $\pp((A^{\epsilon, \gamma})_T^*>\eta, \Gamma_N)$ and $\pp(\Omega\setminus \Gamma_N)$ separately.
On the set $\Gamma_N$,
\[(A^{\epsilon, \gamma})_T^*\leqslant \sqrt{\epsilon}\gamma E_T^{\epsilon *}\sup_{t\leqslant T}
\left|\int_0^t (E_s^{\epsilon})^{-1}\dd W_s\right|\leqslant \sqrt{\epsilon}\gamma N\sup_{t\leqslant T}
\left|\int_0^t (E_s^{\epsilon})^{-1}\dd W_s\right|.\]
Since $\alpha_t^{\epsilon}$ and $\beta_t^{\epsilon}$ are bounded as $\epsilon \rightarrow 0$, it follows
that $\int_0^t[\alpha^{\epsilon}_s-\frac{\epsilon}{2}(\beta^{\epsilon}_s)^2]\dd s$
is  bounded as well, and therefore we omit it for brevity when analyzing 
$E_t^{\epsilon}$.  Based on the above, the stated holds if we can prove that (A)~for all $N\in{\mathbb N}$,
covering the contribution of $\pp((A^{\epsilon, \gamma})_T^*>\eta, \Gamma_N)$,
\begin{equation} \label{step1}
 \lim_{\gamma \rightarrow 0} \limsup_{\epsilon\rightarrow 0}\epsilon \log 
\pp \left(\sqrt{\epsilon}\gamma N\sup_{t\leqslant T}
\left|\int_0^t (E_s^{\epsilon})^{-1}\dd W_s\right|>\eta, \Gamma_N\right)=-\infty, \end{equation}
and~(B), covering the contribution of $\pp(\Omega\setminus \Gamma_N)$,
\begin{equation}\label{step2}
 \lim_{N \rightarrow \infty} \limsup_{\epsilon\rightarrow 0}\epsilon \log 
\pp\left(\sqrt{\epsilon}\sup_{t \leqslant T}\left|\int_0^t \beta_s^{\epsilon} \dd B_s\right|> \log N\right)=-\infty.
\end{equation}

Let us first consider contribution (A). To this end, define
\[\tau:=T \wedge \inf\left \{t\leqslant T: 
\left|\int_0^t (E_s^{\epsilon})^{-1}\dd W_s\right|>\frac{\eta}{\sqrt{\epsilon}\gamma N}\right\}.\]
Then (\ref{step1}) is equivalent to, for all $N\in{\mathbb N}$,
\begin{equation} \label{step3}
 \lim_{\gamma \rightarrow 0} \limsup_{\epsilon\rightarrow 0}\epsilon \log 
\pp \left(\sqrt{\epsilon}\gamma N
\int_0^{\tau} (E_s^{\epsilon})^{-1}\dd W_s\geqslant \eta \:\:
(\text{or} \leqslant -\eta), \Gamma_N\right)=-\infty.
\end{equation}
For $N\in \mathbb{N}$ and $\eta>0$, we define the process $\tilde{E}_t^{\epsilon}$ and
its stochastic exponential $\eee(\tilde{E}^{\epsilon})_t$:
\[\tilde{E}_t^{\epsilon}:=\frac{\eta}{\sqrt{\epsilon}\gamma N^3T}\int_0^{t} (E_s^{\epsilon})^{-1}\dd W_s,\:\:\:\:\eee(\tilde{E}^{\epsilon})_t=\exp\left(\tilde{E}_t^{\epsilon}-\frac{1}{2}\langle\tilde{E}^{\epsilon}\rangle_t\right).\]
Since $\eee(\tilde{E}^{\epsilon})_t$ is a supermartingale, we have
\[\ee\left[\one_{\{\sqrt{\epsilon}\gamma N
\int_0^{\tau} (E_s^{\epsilon})^{-1}\dd W_s\,\geqslant\, \eta,  \Gamma_N\}}\eee(\tilde{E}^{\epsilon})_{\tau}\right]\leqslant 1\]
On the set $\{\sqrt{\epsilon}\gamma N
\int_0^{\tau} (E_s^{\epsilon})^{-1}\dd W_s\geqslant \eta,  \Gamma_N\}$, we have
\[\eee(\tilde{E}^{\epsilon})_{\tau}\geqslant  
\exp\left(\frac{\eta}{\sqrt{\epsilon}\gamma N^3T}\frac{\eta}{\sqrt{\epsilon}\gamma N}
-\frac{1}{2}\left(\frac{\eta}{\sqrt{\epsilon}\gamma N^3T}\right)^2 N^2T\right)
=\exp\left(\frac{\eta^2}{2\epsilon \gamma^2 N^4 T}\right),\]
and consequently
\[\exp\left(\frac{\eta^2}{2\epsilon \gamma^2 N^4 T}\right)
\pp \left(\sqrt{\epsilon}\gamma N
\int_0^{\tau} (E_s^{\epsilon})^{-1}\dd W_s\geqslant \eta, \Gamma_N\right) \leqslant 1.\]
We conclude that the part corresponding to  ``$\geqslant \eta$'' in (\ref{step3}) is valid, but
it is immediately verified that
the part corresponding to ``$\leqslant -\eta$'' in (\ref{step3}) can be addressed in the same way.

We now turn to contribution (B). The validity of (\ref{step2}) can be proved in a similar way by defining the stopping time 
\[\tau':=\inf\left \{t\leqslant T: 
\left|\int_0^t \beta_s^{\epsilon}\dd B_s\right|>\frac{\log N}{\sqrt{\epsilon}}\right\}\] and the process
$\tilde{\beta}_t^{\epsilon}$ and its stochastic exponential $\eee(\tilde{\beta}^{\epsilon})_t$:
\[\tilde{\beta}_t^{\epsilon}:=\frac{\log N}{\sqrt{\epsilon}K^2 T}\int_0^t \beta_s^{\epsilon} \dd B_s,\:\:\:\:\eee(\tilde{\beta}^{\epsilon})_t=\exp\left(\tilde{\beta}_t^{\epsilon}-\frac{1}{2}\langle\tilde{\beta}^{\epsilon}\rangle_t\right),\]
where $K$ is the constant in (\ref{bou}).
\end{proof}

The following result establishes the lower bound of the local LDP. 

\begin{proposition} \label{low}
For every $(\varphi, \nu) \in \mathbb{C}_T \times \mathbb{M}$,
\[\liminf_{\delta\rightarrow 0}\liminf_{\epsilon \rightarrow 0} \epsilon \log \PP
(\rho_T(M^{\epsilon}, \varphi)+d_T(\nu^{\epsilon}, \nu)\leqslant \delta)\geqslant -L_{T}(\varphi, \nu).\] 
\end{proposition}
\begin{proof}
As mentioned in the beginning of this section, only 
the case $(\varphi, \nu) \in \mathbb{H}_T \times \mathbb{M}_T$ such that 
\[\int_{0}^{T} \frac{[\varphi'_{t}-\hat{b}(\nu, \varphi_{t})]^2}{\hat{\sigma}^2(\nu, \varphi_{t})} \DD t < \infty\] needs to be considered. 
If $\inf_{i, x}\sigma^2(i , x) > 0$, then the result is valid due to Proposition \ref{non}.
If $\inf_{i, x}\sigma^2(i , x) = 0$, then we consider $M_{t}^{\epsilon, \gamma}$ as defined in (\ref{nonsingular}).
The idea is that we decompose the probability $\pp\left(\rho_T(M^{\epsilon, \gamma}, \varphi)+d_T(\nu^{\epsilon}, \nu)\leqslant \frac{\delta}{2}\right)$ into 
the sum of 
\begin{equation}
\label{eerste}\pp\left(\rho_T(M^{\epsilon, \gamma}, \varphi)+d_T(\nu^{\epsilon}, \nu)\leqslant \frac{\delta}{2}, \rho_T(M^{\epsilon}, \varphi)+d_T(\nu^{\epsilon}, \nu)\leqslant \delta\right)\end{equation}
and 
\begin{equation}
\label{tweede}\pp\left(\rho_T(M^{\epsilon, \gamma}, \varphi)+d_T(\nu^{\epsilon}, \nu)\leqslant \frac{\delta}{2}, \rho_T(M^{\epsilon}, \varphi)+d_T(\nu^{\epsilon}, \nu)> \delta\right).\end{equation}
Obviously, (\ref{eerste}) is majorized by $\pp( \rho_T(M^{\epsilon}, \varphi)+d_T(\nu^{\epsilon}, \nu)\leqslant \delta)$.
Using the triangle inequality, we find that $\rho_T(M^{\epsilon}, \varphi) \leqslant \rho_T(M^{\epsilon}, M^{\epsilon, \gamma})+\rho_T(M^{\epsilon, \gamma}, \varphi)$. So that
(\ref{tweede}) is majorized by $\pp\left(\rho_T(M^{\epsilon}, M^{\epsilon, \gamma})> \frac{\delta}{2}\right)$.
Hence, $\pp\left(\rho_T(M^{\epsilon, \gamma}, \varphi)+d_T(\nu^{\epsilon}, \nu)\leqslant \frac{\delta}{2}\right)$ is not greater than
\[2 \max\left[\pp( \rho_T(M^{\epsilon}, \varphi)+d_T(\nu^{\epsilon}, \nu)\leqslant \delta), \pp\left(\rho_T(M^{\epsilon}, M^{\epsilon, \gamma})> \frac{\delta}{2}\right)\right].\]
By Lemma \ref{gamma},
\[\lim_{\gamma \rightarrow 0} \limsup_{\epsilon\rightarrow 0}\epsilon \log \pp
\left(\rho_{T}(M^{\epsilon, \gamma}, M^{\epsilon})> \frac{\delta}{2} \right)=-\infty,\]
and, as a result,
 \begin{eqnarray*}\lefteqn{\hspace{-1cm}\liminf_{\delta\rightarrow 0}\lim_{\gamma \rightarrow 0} \liminf_{\epsilon\rightarrow 0}\epsilon \log \pp\left(\rho_T(M^{\epsilon, \gamma}, \varphi)+d_T(\nu^{\epsilon}, \nu)\leqslant \frac{\delta}{2}\right)}\\
&\leqslant&\liminf_{\delta\rightarrow 0} \liminf_{\epsilon\rightarrow 0}\epsilon \log  \pp( \rho_T(M^{\epsilon}, \varphi)+d_T(\nu^{\epsilon}, \nu)\leqslant \delta).\end{eqnarray*}
Next we compute the term on the left-hand side of the above inequality.
Since $M^{\epsilon, \gamma}$ meets the conditions in Lemma \ref{non}, $(M^{\epsilon, \gamma}, \nu^{\epsilon})$ satisfies the inequality (\ref{ustar}). 
Then for every $\gamma>0$,  we obtain
\[ \liminf_{\epsilon \rightarrow 0}\epsilon \log\pp\left(\rho_T(M^{\epsilon, \gamma}, \varphi)+d_T(\nu^{\epsilon}, \nu)\leqslant \frac{\delta}{2}\right) 
 \geqslant -K_{Q, {u^*}} \frac{\delta}{2} d -\tilde{I}_T(\nu)
 -\frac{1}{2}\int_{0}^{T} \frac{[\varphi'_{s}-\hat{b}(\varphi_{s},\nu)]^2}{\hat{\sigma}^2(\varphi_{s}, \nu)+\gamma^2} \DD s. \]
By the monotone convergence theorem (recall the convention $0/0=0$),
\[ \frac{1}{2}\int_{0}^{T} \frac{[\varphi'_{s}-\hat{b}(\varphi_{s},\nu)]^2}{\hat{\sigma}^2(\varphi_{s}, \nu)+\gamma^2} \DD s
\rightarrow \frac{1}{2}\int_{0}^{T} \frac{[\varphi'_{s}-\hat{b}(\varphi_{s},\nu)]^2}{\hat{\sigma}^2(\varphi_{s}, \nu)} \DD s
=I_T(\varphi,\nu), \text{ as } \gamma \rightarrow 0\]
which implies that
\[\liminf_{\delta\rightarrow 0}\lim_{\gamma \rightarrow 0} \liminf_{\epsilon\rightarrow 0}\epsilon \log \pp\left(\rho_T(M^{\epsilon, \gamma}, \varphi)+d_T(\nu^{\epsilon}, \nu)\leqslant \frac{\delta}{2}\right) \geqslant  -\tilde{I}_T(\nu)-I_T(\varphi,\nu). \]
We have proven the claim.
\end{proof}

\appendix

\section{Appendix}

\begin{lemma} \label{integrationbyparts}
Let $f(t, i)$ be a continuous function on $[0, T]$ for every $i \in \mathbb{S}$. 
Let $\mu, \nu \in \mathbb{M}$ such that 
$d_T (\mu, \nu) \leqslant \delta$. For any $\gamma>0$ and $[t_1, t_2] \subset [0, T]$, there exists a constant
$C>0$ such that
\begin{equation}
\sup_{i \in \mathbb{S}}\left| \int_{t_1}^{t_2}f(s, i)[K_{\mu}(s, i)-K_{\nu}(s, i)] \DD s \right|\leqslant C\delta+\gamma.
\end{equation}
\end{lemma}
\begin{proof}
We first look at functions $f(t, i)$ that are of bounded variation. By integration by parts, we have
\[\int_{t_1}^{t_2}f(s, i)[K_{\mu}(s, i)-K_{\nu}(s, i)] \DD s
=[\mu(s, i)-\nu(s, i)] f(s, i)|_{t_1}^{t_2}- \int_{t_1}^{t_2}[\mu(s, i)-\nu(s, i)]\DD f(s, i).\]
Then
\begin{eqnarray*}
\left| \int_{t_1}^{t_2}f(s, i)[K_{\mu}(s, i)-K_{\nu}(s, i)] \DD s \right|
&\leqslant& |\mu(t_2, i)-\nu(t_2, i)|| f(t_2, i)|+|\mu(t_1, i)-\nu(t_1, i)|| f(t_1, i)|\\
&&\hspace{2cm}+\:\ \int_{t_1}^{t_2}|\mu(s, i)-\nu(s, i)||\DD f(s, i)|\\
&\leqslant& C_1\delta+C_2 \delta+ TV_f[t_1, t_2] \delta,
\end{eqnarray*}
where $TV_f[t_1, t_2]$ denotes the total variation of $f$ on $[t_1, t_2]$ and $C_1, C_2$ are two positive constants. Since $\mathbb{S}$ has finite elements,
we can find a constant $C$ such that the claim holds. If $f(t, i)$ is only continuous, it can be uniformly approximated by
a continuously differentiable function (see \cite{pur}), that is, for any $\gamma>0$, there exists a continuously differentiable 
function $f^{\gamma}(t, i)$ such that \[\sup_{t\in [t_1, t_2], i \in \mathbb{S}}|f(t,i)-f^{\gamma}(t, i)|<\gamma/2\,(t_2-t_1).\]
Then
\begin{eqnarray*}
\sup_{i \in \mathbb{S}}\left| \int_{t_1}^{t_2}f(s, i)[K_{\mu}(s, i)-K_{\nu}(s, i)] \DD s \right|
&\leqslant& \sup_{i \in \mathbb{S}}\left| \int_{t_1}^{t_2}\left[f(s, i)-f^{\gamma}(s, i)\right][K_{\mu}(s, i)-K_{\nu}(s, i)] \DD s \right|\\
&&\hspace{1cm}+\:\ \sup_{i \in \mathbb{S}}\left| \int_{t_1}^{t_2}f^{\gamma}(s, i)[K_{\mu}(s, i)-K_{\nu}(s, i)] \DD s \right|\\
&\leqslant& \gamma+ C\delta.
\end{eqnarray*}
This finishes our proof.
\end{proof}

For any $u\in \mathbb{U}$, let $Q(u)(t)$ be the transition matrix resulting from the measure change induced by the stochastic exponential $\eee(\tilde{N}^{\epsilon})$. It is known, see Proposition 11.2.3 in Bielecki and Rutkowski \cite{bielecki}, that 
\[Q(u)(t)_{ij}=Q_{ij}\frac{u(t, j)}{u(t, i)}\:\:\mbox{if $i\not = j$;}\:\:\:
Q(u)(t)_{ii}=-\sum_{j\neq i}Q_{ij}\frac{u(t, j)}{u(t, i)}.\]  For a fixed $t$, we suppress this $t$, so as to make the notation more compact.
In matrix notation, (where we throughout write $\diag(u)$ to denote the diagonal matrix with entries $u_i\delta_{ij}$) we have 
\begin{equation}\label{eq:qqu}
Q(u)=\diag(u)^{-1}Q\diag(u)-\diag(u)^{-1}\diag(Qu).
\end{equation}

\begin{lemma} \label{invariant}
Let $\nu$ be a d-dimensional vector such that $\sum_{i=1}^{d}\nu(i)=1$ and $\nu(i) > 0$. Let $u^*\in U$ be an optimizer of 
\[
\inf_{u \in U}\sum_i\frac{(Q u)_i}{u_i}\nu_i=\inf_{u\in U}\nu^{\rm T}\diag(u)^{-1}Qu.
\]
Then $\nu$ is the unique invariant vector of the transition matrix $Q(u^*)$. 
\end{lemma}

\begin{proof}
To find a minimizing $u^*=u(\nu)$ (where it observed that minimizers are not necessarily unique) for the above problem, we first note that all $u^*_i>0$. Hence the minimizer solves the system of first order conditions. Differentiation with respect to $u_k$ yields ($e_k$ denoting the $k$-th basis vector)
\[
-\nu^{\rm T} \diag(u)^{-1}e_ke_k^{\rm T} \diag(u)^{-1}Qu+\nu^{\rm T} \diag(u)^{-1}Qe_k=0.
\]
In vector notation these equations can be conveniently summarized as
\begin{equation}\label{eq:nudq}
\nu^{\rm T} \diag(u)^{-1}\left(-\diag(u)^{-1}\diag(Qu)+Q\right)=0.
\end{equation}
From \eqref{eq:qqu} we deduce by commutation of diagonal matrices the relation
\[
Q(u)=\diag(u)^{-1}(Q-\diag(u)^{-1}\diag(Qu))\diag(u),
\]
and hence
\[
\diag(u)Q(u)\diag(u)^{-1}=Q-\diag(u)^{-1}\diag(Qu).
\]
We can therefore rewrite \eqref{eq:nudq} as
\[
\nu^{\rm T} Q(u)\diag(u)^{-1}=0.
\]
It follows that $\nu^{\rm T}  Q(u^*)=0$. Since $\nu(i)>0$, $\nu$ is the unique invariant vector of $Q(u^*)$. 
\end{proof}

\end{document}